\documentclass[a4paper,11pt,oneside]{article}

\usepackage{amsmath,amssymb}
\usepackage{epsfig}
\usepackage[utf8]{inputenc}
\usepackage{psfrag,textcomp}
\usepackage{algorithm}
\usepackage{algpseudocode}
\usepackage{color}
\usepackage{mathtools}
\usepackage{graphicx,import}
\usepackage{pgfplots}
\usepackage{caption} 
\usepackage{subcaption}
\usepackage{bm}
\usepackage{empheq}

\usepackage{geometry}
\newtheorem{thm}{Theorem}[section]

\newtheorem{definition}[thm]{Definition}
\newtheorem{proposition}[thm]{Proposition}

\newtheorem{proof}{Proof}
\newtheorem{remark}[thm]{Remark}

\title{A Nonintrusive Distributed Reduced Order Modeling Framework for nonlinear structural mechanics -- application to elastoviscoplastic computations}

\author{Fabien Casenave\thanks{SafranTech, Rue des Jeunes Bois - Ch\^ateaufort, 78114 Magny-les-Hameaux, France}, Nissrine Akkari${}^{*}$, Felipe Bordeu${}^{*}$, Christian Rey${}^{*}$, David Ryckelynck\thanks{MINES ParisTech, PSL - Research University, Centre des mat\'eriaux, CNRS UMR-7633, Evry, France}}

\newlength\figureheight
\newlength\figurewidth

\begin{document}
\maketitle

\begin{abstract}
In this work, we propose a framework that constructs reduced order models for nonlinear structural mechanics in a nonintrusive fashion, and can handle large scale simulations.
We identify three steps that are carried out separately in time, and possibly on different devices: (i) the production of high-fidelity solutions by a commercial software, (ii) the \textit{offline} stage of the model reduction and (iii) the \textit{online} stage where the reduced order model is exploited. The nonintrusivity assumes that only the displacement field solution is known, and relies on operations on simulation data during the \textit{offline} phase by using an in-house code.
The compatibility with a new commercial code only needs the implementation of a routine converting the mesh and result format into our in-house data format.
The nonintrusive capabilities of the framework are demonstrated on numerical experiments using commercial versions of the finite element softwares Zset and Ansys Mechanical. The nonlinear constitutive equations are evaluated by using the same external plugins as for Zset or Ansys Mechanical.
The large scale simulations are handled using domain decomposition and parallel computing with distributed memory. The features and performances of the framework are evaluated on two numerical applications involving elastoviscoplastic materials: the second one involves a model of high-pressure blade, where the framework is used to extrapolate cyclic loadings in 6.5 hours, whereas the reference high-fidelity computation would take 9.5 days.
\end{abstract}

\section{Introduction}
\label{sec:intro}

A performance race is currently underway in the aircraft engine industry. The materials of some critical parts of the engine are pushed to the limit of their strength to increase as much as possible the efficiency of the propulsion. In particular, the high-pressure turbine blades, which are located directly downstream from the chamber of combustion, undergo extreme thermal loading. The possible causes of failure for these turbines include temperature creep rupture and high-cycle fatigue~\cite{mazur2005474, cowles1996}. Many efforts are spent to increase the strength of these turbines, with the use of  thermal barriers~\cite{coatings}, advanced superalloys~\cite{superalloys} and complex internal cooling channels~\cite{coolingchannels1, coolingchannels2}, see Figure~\ref{fig:gaturbine} for a representation of a high-pressure turbine blade. The lifetime prediction of such complex parts is a very demanding computational task: the meshes are very large to account for small structures such as the cooling channels, the constitutive laws are strongly nonlinear and involve a large number of internal variables, and more importantly, a large number of cycles has to be simulated. Indeed, failures come from local structural effects whose evolution cannot be predicted without computing potentially hundreds of cycles.

\begin{figure}[h]
  \centering
  \includegraphics[width=0.25\textwidth]{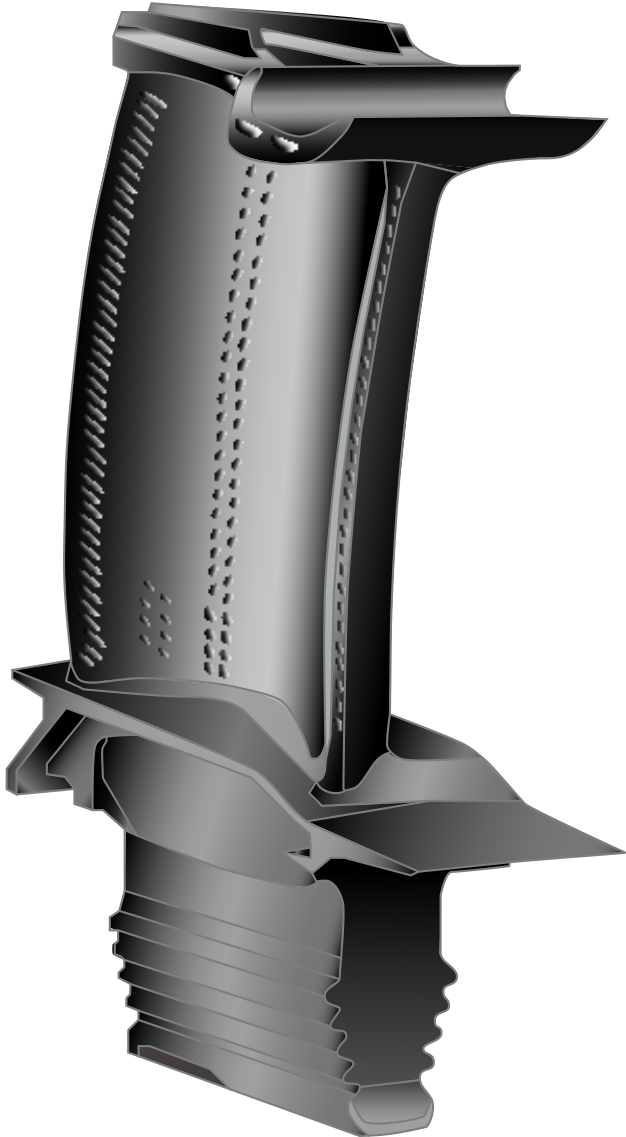}
  \caption{Schematic of a high-pressure turbine blade~\cite{gaturbine}. Internal cooling channels allow the cool air to exit and protect the outer surface of the blade with a cooling film.}
  \label{fig:gaturbine}
\end{figure}

Like the lifetime prediction of turbine blades, many industrial needs involve very intensive numerical procedures, where the approximations of partial differential equations are solved. Efforts are spent to mitigate the runtime issues. In the context of nonlinear structural mechanics, we consider the conjugation 
of parallel computing with distributed memory to accelerate the resolution of large problems with Reduced Order Modeling, which learns about the physical operators from a solution set computed beforehand to accelerate further computations.

In what follows, we consider \textit{a posteriori} reduced order modeling methods, where the equations of the physics are solved in the \textit{online} stage.
These methods  require the resolution of high-fidelity solutions, in general using a finite element code. In this context, the reduced model consists in solving the same Galerkin problem as the high-fidelity code, but on a reduced order basis instead of a finite element basis.
These methods are intrinsically intrusive: when applied using a given finite element code, modifications have to be implemented up to the assembling routines.
As a consequence, the majority of the contributions to these methods are illustrated using in-house codes.
We refer to the introduction of~\cite{giraldi} for a detailed reflection on the notion of intrusivity of approximation methods for parametric equations, and the references within for methods falling into intrusive and nonintrusive categories.
Following~\cite{giraldi}, we call a reduction method nonintrusive if it only relies on commercial software to compute the high-fidelity solution, and \textit{a posteriori} reduced order modeling methods belong the so-called \textit{Galerkin-like} methods, for which only few nonintrusive contributions are acknowledged.

In the industry, nonintrusive methods to accelerate numerical simulations are widely used. 
Besides the practicality, more important reasons explain the use of nonintrusive methods: all the simulation methodologies are certified using a given version of a code and modifying this code means a new certification procedure; moreover, commercial codes are very often used in design offices to take advantage of assistance contracts provided by the editors of the software. Hence, deriving nonintrusive procedures appears a prerequisite to spread these methods in the industry.
The literature about nonintrusive methods reducing the computational complexity of numerical simulation is vast, but in the large majority, not being \textit{Galerkin-like}, they tend to “forget” the original equations. Concerning statistical and regression learning techniques, we can cite the following reviews~\cite{reviewmachinelearning1, reviewmachinelearning2, reviewmachinelearning3, reviewmachinelearning4,reviewmachinelearning5}.
In the \textit{a posteriori} reduced order modeling , efforts have been spent to mitigate the intrusivity constraints. In~\cite{Casenave2015}, a reduced basis method is proposed for boundary elements methods only requiring the high-fidelity code to multiply the operator by user-provided vectors (approximately using fast multipole method, the matrix operator never being constructed). In~\cite{giraldi}, parametric problems are approximated only assuming we can use the first iteration of the high-fidelity iterative solver. In~\cite{Audouze, Hesthaven}, a reduced order basis is computed, and the coefficients of the reduced solution on this basis are computed using interpolation or regression methods instead of solving the Galerkin method in the \textit{online} phase. In~\cite{chakir}, these coefficients are approximated using a two-grid method and only resorting to the finite-element code in a nonintrusive fashion.

The contribution of this work consists in the construction of a framework able to reduce large scale nonlinear structural mechanics problems in a nonintrusive fashion, namely using commercial codes. The framework uses different algorithms taken from the literature and employs them to treat large scale problems. All the computation procedure is parallel with distributed memory: the computation of the high-fidelity model, 
the reduction routines using in-house distributed python routines, up to the visualization and post-treatment. The nonintrusivity feature is obtained by having coded all the finite-element operation needed in \textit{a posteriori} reduced order modeling procedure in a collection of in-house routines.
A detailed numerical application is provided in Section~\ref{sec:evp}, where we explain how this reduced order framework has been used to tackle the lifetime prediction of large models of high-pressure turbine blades. To the limit of our knowledge, this is the first nonintrusive \textit{a posteriori} reduced order modeling framework making use of parallel computing with distributed memory.

In what follows, we first present in Section~\ref{sec:hfm} the problem of interest: a quasistatic evolution of an elastoviscoplastic body under a time-dependent loading. Then, the \textit{a posteriori} reduced order modeling of this problem is introduced in Section~\ref{sec:rom}. Section~\ref{sec:nonintrusif} presents the framework, with the algorithmic choices made to obtain a nonintrusive parallel procedure. Finally, the features and performance of the framework are illustrated in two numerical experiments involving elastoviscoplastic materials in Section~\ref{sec:num}.

\section{High-fidelity model for nonlinear structural mechanics}
\label{sec:hfm}

Consider a structure, denoted $\Omega$, whose boundary $\partial\Omega$ is partitioned as $\partial\Omega=\partial\Omega_D\cup\partial\Omega_N$ such that $\partial\Omega_D\cap\partial\Omega_N=\emptyset$, see Figure~\ref{fig:body}.
\begin{center}
\def\svgwidth{0.4\textwidth}
\begin{figure}[h]
  \centering
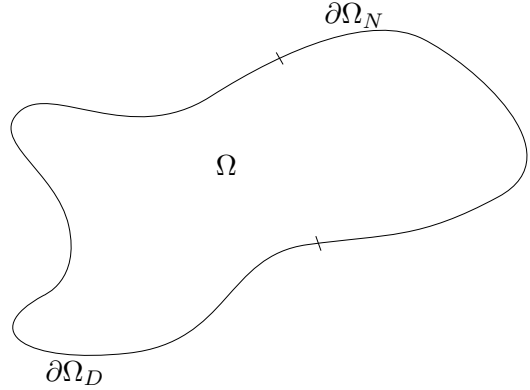
  \caption{Illustration of the structure of interest $\Omega$.}
\label{fig:body}
\end{figure}
\end{center}
The structure undergoes a quasi-static time-dependent loading, composed of Dirichlet boundary conditions on $\partial\Omega_D$ in the form of prescribed zero displacement and Neumann boundary conditions  on $\partial\Omega_N$ in the form of a prescribed traction $T_N$, as well as a volumic force $f$. 
From~\cite[Equation (1.1)]{bonnet2014finite}, under the small deformation assumption, the evolution of the structure over the time interval $t\in[0,T]$ is governed by equations

\begin{subequations}
\begin{align}
&\epsilon(u)=\frac{1}{2}\left(\nabla u+ \nabla^T u\right)&\rm{in~}\Omega\times[0,T]&\quad&\rm{(compatibility),}\label{eq:r1}\\
&\rm{div}\left(\sigma\right)+f=0&\rm{in~}\Omega\times[0,T]&\quad&\rm{(equilibrium),}\label{eq:r2}\\
&\sigma=\sigma(\epsilon(u),y)&\rm{in~}\Omega\times[0,T]&\quad&\rm{(constitutive~law),}\label{eq:r3}\\
&u=0&\rm{in~}\partial\Omega_D\times[0,T]&\quad&\rm{(prescribed~zero~displacement),}\label{eq:r4}\\
&\sigma\cdot n=T_N&\rm{in~}\partial\Omega_N\times[0,T]&\quad&\rm{(prescribed~traction),}\label{eq:r5}\\
&u=0,y=0&\rm{in~}\Omega~{\rm at~t=0}&\quad&\rm{(initial~condition),}\label{eq:r6}
\end{align}
\end{subequations}
where $u$ is the unknown displacement, $\epsilon$ is the linear strain tensor, $\sigma$ is the Cauchy stress tensor, $y$ denotes the internal variables of the constitutive law, and
$n$ is the exterior normal on $\partial\Omega$. 

Due to the nonintrusive nature of our framework, we do not restrict our high-fidelity model to any particular numerical method: we simply need some discretized approximation of the solution to~\eqref{eq:r1}-\eqref{eq:r6} on some mesh of $\Omega$. However, we present a classical numerical method, using a Galerkin method on a finite element space and a Newton algorithm. Actually, our reduced model will use these tools.

Define $H^1_0(\Omega)=\{v\in L^2(\Omega)|~\frac{\partial v}{\partial x_i}\in L^2(\Omega),~1\leq i\leq 3{\rm~and~}v|_{\partial\Omega_D}=0\}$. Denote $\{\varphi_i\}_{1\leq i\leq N}$, a finite element basis whose span, denoted $\mathcal{V}$, constitutes an approximation of $H^1_0(\Omega)^3$.
A discretized weak formulation obtained by reaction elimination from \eqref{eq:r1}-\eqref{eq:r6} writes:
find $u\in \mathcal{V}$ such that for all $v\in \mathcal{V}$,
\begin{equation}
\int_{\Omega}\sigma(\epsilon(u),y):\epsilon(v)=\int_{\Omega}f\cdot v+\int_{\partial\Omega_N}T_N\cdot v,
\end{equation}
that we denote for concision $\mathcal{F}\left(u\right)=0$. A Newton algorithm can be used to solve the nonlinear global equilibrium problem $\mathcal{F}\left(u\right)=0$ at each time step
\begin{equation}
\label{eq:Newton}
\displaystyle\frac{D\mathcal{F}}{Du}\left(u^k\right)\left(u^{k+1}-u^{k}\right)=-\mathcal{F}\left(u^k\right),
\end{equation}
where $u^k$ is the k-th iteration of the discretized displacement field for the considered time-step (the element $u^k\in\mathbb{R}^N$ is identified with $\displaystyle\sum_{i=1}^Nu^k_i\varphi_i\in\mathcal{V}$),
\begin{equation}
\label{eq:TgtMat}
\displaystyle{\frac{D\mathcal{F}}{Du}\left({u}^k\right)}_{ij}=\int_{\Omega}\epsilon\left(\varphi_j\right):\mathcal{K}\left(\epsilon(u^k),y\right):\epsilon\left(\varphi_i\right),
\end{equation}
where $\mathcal{K}\left(\epsilon(u^k),y\right)$ is the local tangent operator, and
\begin{equation}
\label{eq:problem}
\displaystyle {\mathcal{F}\left({u}^k,y\right)}_i=\int_\Omega\sigma\left(\epsilon({u}^k),y\right):\epsilon\left(\varphi_i\right)-\int_\Omega f\cdot\varphi_i-\int_{\partial\Omega_N}T_N\cdot\varphi_i.
\end{equation}
The Newton algorithm stops when the norm of the residual divided by the norm of the external forces vector is smaller than a user-provided tolerance, denoted $\epsilon^{\rm HFM}_{\rm Newton}$.

In Equation~\eqref{eq:Newton}, $f$, $T_N$, $u^k$ and $y$ from~\eqref{eq:problem} are known and enforce the time-dependence of the solution. Notice that depending on the constitutive law, the computation of the functions $\displaystyle\left({u}^k,y\right)\mapsto \sigma\left(\epsilon({u}^k),y\right)$ and $\displaystyle\left({u}^k,y\right)\mapsto \mathcal{K}\left(\epsilon(u^k),y\right)$ can involve complex ordinary differential equations (ODEs), and therefore advanced numerical procedures. 
The constitutive laws considered in our applications will be made explicit in Section~\ref{sec:num}. In this work, these variables are assumed to be computed by an external software.

\section{Reduced Order Modeling}
\label{sec:rom}


In a reduced order model context, the considered high-fidelity model may also depend on some parameter $\mu$, and the collection of all the computed solutions (temporal and/or parametric) is referred to as snapshots. We denote $u_s$, $1\leq s\leq N_c$ the snapshots at our disposal, where $N_c$ denotes their number.
We consider an \textit{a posteriori} projection-based reduced order model: a Galerkin method is no longer written on a finite element basis $(\varphi_i)_{1\leq i\leq N}$, but rather on a reduced order basis $(\psi_i)_{1\leq i\leq n}$, with $n\ll N$, constructed via some operation of the snapshots. 
We assume the problem at hand to be reducible, in the sense that the obtained approximation is accurate with $n$ small enough for practical computational gains to be obtained.
We define the matrix $\Psi\in\mathbb{R}^{N\times n}$, whose columns contain the vectors of the reduced order basis. These reduced order methods are usually composed of an \textit{offline} stage: the learning phase where expensive operators are allowed on the high-fidelity model, and an \textit{online} stage: the exploitation phase where an approximation of the high-fidelity model has to be computed in a complexity which is independent of the dimension $N$ of the finite element basis of the high fidelity problem.

\subsection{\textit{Online} stage: the reduced problem}
\label{sec:redprob}

Once the reduced order basis is determined, the reduced Newton algorithm can be constructed
\begin{equation}
\label{eq:reducedNewton}
\displaystyle\frac{D\mathcal{F}}{Du}\left(\hat{u}^k\right)\left(\hat{u}^{k+1}-\hat{u}^{k}\right)=-\mathcal{F}\left(\hat{u}^k\right),
\end{equation}
where $\hat{u}^k$ is the k-th iteration of the reduced displacement field for the considered time-step -- the element $\hat{u}^k\in\mathbb{R}^n$ is identified with $\displaystyle\sum_{i=1}^n\hat{u}^k_i\psi_i\in\hat{\mathcal{V}}:=\rm{Span}\left(\psi_i\right)_{1\leq i\leq n}$,
\begin{equation}
\label{eq:redTgtMat}
\displaystyle{\frac{D\mathcal{F}}{Du}\left(\hat{u}^k\right)}_{ij}=\int_{\Omega}\epsilon\left(\psi_j\right):\mathcal{K}\left(\epsilon(\hat{u}^k),y\right):\epsilon\left(\psi_i\right)
\end{equation}
and
\begin{equation}
\label{eq:redProblem}
\displaystyle {\mathcal{F}\left(\hat{u}^k,y\right)}_i=\int_\Omega\sigma\left(\epsilon(\hat{u}^k),y\right):\epsilon\left(\psi_i\right)-\int_\Omega f\cdot\psi_i-\int_{\partial\Omega_N}T_N\cdot\psi_i.
\end{equation}
The reduced Newton algorithm stops when the norm of the reduced residual divided by the norm of the reduced external forces vector is smaller than a user-provided tolerance, denoted $\epsilon^{\rm ROM}_{\rm Newton}$.
We say that the \textit{online} stage is efficient if~\eqref{eq:reducedNewton} can be constructed in computational complexity independent of $N$.

\subsection{\textit{Offline} stage: the three-step procedure}
\label{sec:3stepProc}

In the light of the context presented so far, we can define the following three-step procedure common to \textit{a posteriori} projection-based reduced order models:
\begin{enumerate}
\item(Data generation) This step corresponds to the generation of the snapshots by solving the high-fidelity model. These snapshots represent the data of our reduced order modeling: we search for a small dimension functional subspace that describe correctly this data. Besides, the global behavior of the structure $\Omega$ is only accessible through the knowledge of this data.

\item(Data compression) This step corresponds to the generation of the reduced order basis $(\psi_i)_{1\leq i\leq n}$. Usually, it consists of some post-treatment of the snapshots, by looking for a hidden low-rank structure. The data is compressed into the reduced order basis.

\item(Operator compression) This step corresponds to the additional treatment needed to guarantee the efficiency of the \textit{online} stage, by pre-treatment of the computationally demanding integration over $\Omega$ and $\partial\Omega_N$. 
Notice that in Equation~\eqref{eq:reducedNewton}, without additional approximation, the numerical integration step will in practice strongly limit the efficiency of the ROM (no interesting speedup with respect to the computation of the high-fidelity model will be obtained in practice).
The need to construct such approximation does not depend on the complexity of the high-fidelity model (like linear vs. nonlinear), but on the type of parameter-dependence of the problem. 
This step is actually needed for all classes of problems reduced by projection-based methods.

In the favorable case of a linear problem with an affine dependence in the parameter $\mu$, for instance $A_\mu u =b$, where $A_\mu = A_0+\mu A_1$, the reduced problem $\Psi^T A_\mu \Psi \hat{u} =\Psi^T b$ is not assembled in the \textit{online} phase, but the matrices $\Psi^T A_0\Psi$ and $\Psi^T A_1 \Psi$, and the vector $\Psi^T b$, can be precomputed in the \textit{offline} stage so that the reduced problem can be constructed efficiently, without approximation, by summing two small matrices. The construction of $\Psi^T A_0\Psi$ and $\Psi^T A_1 \Psi$ in the \textit{offline} stage does constitute the operator compression step.

On the contrary, there exist linear problems that require an additional approximation, and nonlinear problem for which the operation compression step can be carried-out exactly. In the former case, take $A_\mu u=b$ with $\displaystyle A_{ij}=\int_{\Omega} \nabla\left(g(x,\mu)\varphi_j(x)\right)\cdot\nabla \varphi_i(x)$ and $b_i=\int_{\Omega}f(x)\varphi_i(x)$, where $u$ is the unknown, $f$ some known loading and $g(x,\mu)$ a known function whose dependencies in $x$ and $\mu$ cannot be separated: the previous trick consisting of precomputing some reduced matrices cannot be applied, and a treatment is needed to, for instance, approximately separate the dependencies in $x$ and $\mu$ of $g$ as $g(x,\mu)\approx\sum_{k=1}^d g^a_k(x)g^b_k(\mu)$. Then, $\Psi^T A_\mu \Psi\approx 
\sum_{k=1}^d g^b_k(\mu) A_k$ where $\displaystyle \left(A_k\right)_{ij}=\int_{\Omega} \nabla\left(g^a_k(x)\varphi_j(x)\right)\cdot\nabla \varphi_i(x)$, so that the efficiency of the \textit{online} stage is recovered; the Empirical Interpolation Method has been introduced in~\cite{EIM1,EIM2} for this matter.
An example of linear problem in harmonic aeroacoustics with nonaffine dependence in the frequency is available in~\cite{Casenave2015}. Some nonlinearities can be treated without approximation, for example, the advection term is fluid dynamics (with a ROM based on a Galerkin method in our context) only requires the precomputation of an order-3 tensor in the form $\displaystyle \int_{\Omega}\psi_i\cdot\left(\psi_j\cdot\nabla\right)\psi_k$, $1\leq i,j,k\leq n$, see~\cite{stablePOD} for a reduction of the nonlinear Navier-Stokes equations, with an exact operator compression step. Other examples can be found in structural dynamics with geometric nonlinearities, where order-2 and -3 tensors can be precomputed, see~\cite[Section 3.2]{kuether2014nonlinear} and~\cite{mignolet2013review}.
\end{enumerate}

We now present the algorithms chosen to tackle each of the presented step in a nonintrusive and distributed fashion.

\section{A distributed nonintrusive framework}
\label{sec:nonintrusif}

The framework consists in a set of python classes. The parallel processing is performed by MPI for Python (mpi4py, see~\cite{mpi4py3,mpi4py2,mpi4py1}) and the framework can be used indifferently
 in sequential or in parallel. In the latter case, we suppose that the high-fidelity commercial code provides the mesh, loading and solutions independently for a collection of subdomains $\Omega_l$, $1\leq l\leq n_d$, which constitute a partition of the structure $\Omega$: $\Omega_l\cap \Omega_{l'}=\emptyset$ for all $l\neq l'$, and $\displaystyle\bigcup_{l=1}^{n_d}\Omega_l=\Omega$, see Figure~\ref{fig:bodyDD}. In our application, $\Omega_l$ are the subdomains used in a domain decomposition method, but they can come from any suddivision of the high-fidelity solutions.
%

\begin{center}
\def\svgwidth{0.4\textwidth}
\begin{figure}[h]
  \centering
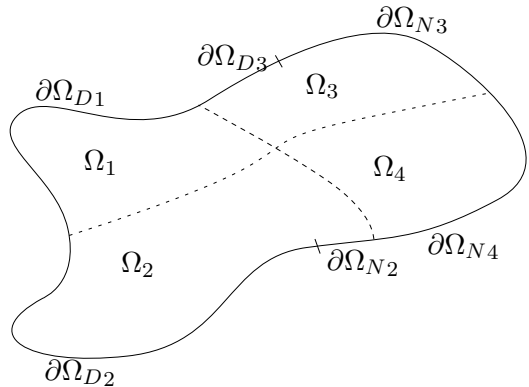
\vspace{0.2cm}
  \caption{Illustration of the structure of interest in a domain decomposition context.}
\label{fig:bodyDD}
\end{figure}
\end{center}

In this section, we explain how the steps of reduced order modeling described in the previous section are carried-out in a nonintrusive and distributed fashion in our framework.

\subsection{Data generation}

In our nonintrusive context, we recall that we restrict ourselves to the use of commercial software. Depending on the software, some data may need to be reconstructed. For instance, with Zset~\cite{zset}, we dispose of $\sigma$, $\epsilon$ and all the internal variables on all the integration points of the mesh. With Ansys Mechanical, these dual quantities are interpolated at the vertices of the mesh. To be as general as possible, we suppose that we only dispose of the displacement field, solution to~\eqref{eq:r1}-\eqref{eq:r6}. We can reconstruct the complete history of all dual quantities at all the integration points using a constitutive law solver, based on the knowledge of displacement field. To do so, we interfaced our python library to the constitutive law API Umat~\cite{abaqus}, making possible the use of an external reusable plugin.
The mesh is converted from the format of the considered code into an in-house unstructured mesh format.
The loading is converted in the form of assembled external forces vectors. To do so, an in-house finite element assembler has been implemented in python as well.

\subsection{Data compression}

Once we dispose of the snapshots $u_s$, $1\leq s\leq N_c$, the reduced order basis can be computed.
We use the snapshot POD (see Algorithm~\ref{algo:snapshotPOD} and~\cite{POD1, POD2}), for its good performance in a distributed framework.

\begin{definition}[$\epsilon$-truncation of the eigendecomposition of a matrix]
It consists in the $n$ largest eigenvalues $\lambda_i$, $1\leq i\leq n$, and associated eigenvectors $\xi_i$, $1\leq i\leq n$ of a square matrix, such that  $n = \max\left(n_1,n_2\right)$, where $n_1$ and $n_2$ are respectively the smallest integers such that $\displaystyle\sum_{i=1}^{n_1} \lambda_i \geq \left(1-\epsilon^2\right)\sum_{i=1}^{N_c} \lambda_i$ and $\lambda_{n_2}\leq \epsilon^2 \lambda_{0}$.
\end{definition}

\begin{algorithm}[h!]
	\caption{Data compression by snapshot POD.}
	\label{algo:snapshotPOD}
\begin{algorithmic}[1]
	\State {Choose a tolerance $\epsilon_{\rm POD}$}
	\State{Compute the correlation matrix $\displaystyle C_{i,j}=\int_{\Omega}u_i\cdot u_j$, $1\leq i,j\leq N_c$}
	\State {Compute the $\epsilon_{\rm POD}$-truncated eigendecomposition of $C$: $\xi_i$ and $\lambda_i$, $1\leq i\leq n$}
           \State{Compute the reduced order basis $\displaystyle\psi_i(x)=\frac{1}{\sqrt{\lambda_i N_c}}\sum_{j=1}^{N_c}u_j(x){\xi_i}_j $, $1\leq i\leq n$}
\end{algorithmic}
\end{algorithm}

Each subdomain can construct its local correlation matrix without any interprocess communication: $\displaystyle C^l_{i,j}=\int_{\Omega_l}u_i|_{\Omega_l}\cdot u_j|_{\Omega_l}$, $1\leq l\leq n_d$. Then, a distribution and reduction operation is carried-out for each process to dispose of the complete correlation matrix $C=\sum_{l=1}^{n_d}C^l$. Each process can compute the same $\epsilon$-truncated eigendecomposition of $C$ and construct the restriction of the global POD modes to their subdomain as $\displaystyle\psi_i|_{\Omega_l}(x)=\frac{1}{\sqrt{\lambda_i N_c}}\sum_{j=1}^{N_c}u_j|_{\Omega_l}(x){\xi_i}_j $, $1\leq i\leq n$, $1\leq l\leq n_d$.
In distributed mode, this step requires an interprocess communication of type all-to-all, for an object of size $\frac{1}{2}N_c\left(N_c+1\right)$. In our type of problems, the number of computed snapshots $N_c$ is usually much smaller than the number of dofs $N$, which means that the interprocess communication time is very limited with respect to the rest of the execution time. For instance, for $N_C=1000$ snapshots, the packages to be communicated have a volume of approximately 4~MB.
The construction of the $L^2(\Omega)$-scalar product matrices is carried-out by our in-house python library.

\subsection{Operator compression}

Among the available methods to treat the operator compression step, we chose to construct a reduced quadrature with positive weights, as described below.

In the finite element method, the integrals in the high-fidelity model~\eqref{eq:Newton} are computed using (in general exact) quadrature formulas. Applying the chosen quadrature to the reduced internal forces vector yields
\begin{equation}
\label{eq:intForcesQuadrature}
\hat{F}^{\rm int}_i(t)=\sum_{e\in E}\sum_{k=1}^{n_e}\omega_k\sigma\left(\epsilon(\hat{u}),y\right)(x_k,t):\epsilon\left(\psi_i\right)(x_{k}),
\end{equation}
where $E$ denotes the set of elements of the mesh, $n_e$ denotes the number of integration points for the element $e$, $\omega_k$ and $x_k$ are the integration weights and points of the considered element, where the stress tensor $\sigma\left(\epsilon(\hat{u}),y\right)$ for the considered reduced solution $\hat{u}$ and internal variables $y$ is seen as a function of space and time and where we recall that $\Psi_i$, $1\leq i\leq n$, denote the vectors of the reduced order basis. We denote also $N_G$ the total number of integration points.
The Energy Conserving Sampling and Weighting (ECSW, see~\cite{ECSW1, ECSW2, ECSW3}) and Empirical Cubature Method (ECM, see~\cite{ECM}) look for approximations of~\eqref{eq:intForcesQuadrature} in the form  of reduced quadrature. In ECSW, a reduced mesh composed of a small number of elements is considered, while in ECM, a reduced set of integration points is selected globally. In both cases, the associated integration weights are taken positive, as was initially proposed in the computer graphics community~\cite{an2008optimizing}, so that the reduced quadrature can be written
\begin{equation}
\label{eq:intForcesQuadrature2}
\hat{F}^{\rm int}_i(t):=\int_\Omega\sigma\left(\epsilon(\hat{u}),y\right)(x,t):\epsilon\left(\psi_i\right)(x)\approx\sum_{k'=1}^{d}\hat{\omega}_{k'}\sigma\left(\epsilon(\hat{u}),y\right)(\hat{x}_{k'},t):\epsilon\left(\psi_i\right)(\hat{x}_{k'}), 1\leq i\leq n,
\end{equation}
where $\hat{\omega}_{k'}>0$ and $\hat{x}_{k'}$ are respectively the reduced integration weights and points, where $d$ is the number of terms in the summation which is now called length of the reduced quadrature.

\begin{remark}[Positivity of the weights]
As noted in~\cite[Remark 2]{ECSW1} for ECSW and~\cite[Section 1.2.2-2]{ECM} for ECM (originally motivated in~\cite{an2008optimizing}), the positivity of the weights of the reduced quadrature preserves the spectral properties of the operator associated with the physical problem: suppose that for all $\mathcal{O}\subset\Omega$, all $v\in H^1_0(\Omega)^3$ such that $v\neq 0$ and all $y$,
$\displaystyle\int_{\mathcal{O}}\epsilon\left(v\right):\mathcal{K}\left(\epsilon(v),y\right):\epsilon\left(v\right)>0$. Then for all $x\in\Omega$, $\epsilon\left(v\right):\mathcal{K}\left(\epsilon(v),y\right):\epsilon\left(v\right)(x)>0$. Consider a reduced quadrature with positive weights defined by $(\hat{\omega}_{k'}, \hat{x}_{k'})_{1\leq k'\leq d}$ and a vector of internal variables $y$, and denote $\displaystyle\mathbb{R}^{n\times n}\ni\frac{D\hat{\mathcal{F}}}{Du}(\hat{v})_{ij}:=\sum_{k'=1}^d\hat{\omega}_{k'}\left(\epsilon\left(\psi_j\right):\mathcal{K}\left(\epsilon(\hat{v}),y\right):\epsilon\left(\psi_i\right)\right)(\hat{x}_{k'})$, the global reduced tangent matrix~\eqref{eq:redTgtMat} computed with a reduced quadrature. For all $\hat{v}\in\hat{\mathcal{V}}$ (associated with an element of $\mathbb{R}^n$), $\displaystyle \hat{v}^t \left(\frac{D\hat{\mathcal{F}}}{Du}(\hat{v})\right)\hat{v}=\sum_{k'=1}^d\underbrace{\hat{\omega}_{k'}}_{>0}\underbrace{\epsilon\left(\hat{v}\right):\mathcal{K}\left(\epsilon(\hat{v}),y\right):\epsilon\left(\hat{v}\right)(\hat{x}_{k'})}_{>0}>0$, which means that the global reduced tangent matrix computed with a reduced quadrature is positive definite.
\end{remark}

We now explain how such reduced quadrature can be computed.
Consider the computation of the components of the reduced internal forces when constructing the reduced problem~\eqref{eq:reducedNewton}: $\displaystyle\int_{\Omega}\sigma\left(\epsilon(\hat{u}),y\right):\epsilon\left(\psi_i\right),~1\leq i\leq n$, where we recall that at convergence of the Newton algorithm, $\hat{u}$ is the reduced solution. The goal is to construct a reduced quadrature formula that can accurately approximate this quantity on the snapshots at our disposal, namely  $\displaystyle\int_{\Omega}\sigma\left(\epsilon(u_s),y\right):\epsilon\left(\psi_i\right),~1\leq i\leq n,~1\leq s\leq N_c$. First denote $f_q:=\sigma\left(\epsilon(u_{(q//n)+1}),y\right):\epsilon\left(\psi_{(q\%n)+1}\right)$, where $//$  and $\%$ denote respectively the quotient and the remainder of the Euclidean division. Then denote $\mathcal{Z}$ a subset of $[1;N_G]$ of size $d$ and $J_{\mathcal{Z}}\in\mathbb{R}^{nN_c\times d}$ and $b\in\mathbb{N}^{nN_c}$ such that for all $1\leq q\leq nN_c$ and all $1\leq q'\leq N_G$, 
\begin{equation}
J_{\mathcal{Z}} = \Bigg(f_q(x_{q'})\Bigg)_{1\leq q\leq nN_c,~q'\in\mathcal{Z}},\qquad b = \left(\int_{\Omega}f_q\right)_{1\leq q\leq nN_c}.
\end{equation}

The notion of a convenient reduced quadrature formula implies a trade-off between speed of evaluation and accuracy of the approximation: the fewer reduced integration points we keep, the faster the reduced quadrature is computed but potentially the approximation error is large. On the contrary, if we keep all the integration points, the quadrature is exact but no reduction in the execution time is obtained.
The problem of finding the best reduced quadrature formula of length $d$ for the reduced internal forces vector can be written (c.f.~\cite[Equation (23)]{ECM})
\begin{equation}
\label{eq:optPb}
\left(\hat{\omega},\mathcal{Z}\right)=\arg\underset{\hat{\omega}'>0,\mathcal{Z}'\subset[1;N_G]}{\min}\left\|J_{\mathcal{Z}'}\hat{\omega}'-b\right\|,
\end{equation}
where $\left\|\cdot\right\|$ stands for the Euclidean norm. As indicated in~\cite[Section 5.3]{ECSW1}, citing~\cite{amaldi}, taking the length of the reduced quadrature formula in the objective function of the optimization leads to a NP-hard problem. 
Approximating the solution to this optimization problem has been intensively addressed by the signal processing community, where different variants of greedy algorithms have been proposed. In~\cite{ECSW1} a Sparse NonNegative Least-Squares (NNLS) algorithm, proposed in~\cite{peharz}, is used
to reach a solution in reasonable time.
Here, we consider a modification of the Nonnegative Orthogonal Matching Pursuit algorithm, see~\cite[Algorithm 1]{fastNNOMP} and Algorithm~\ref{algo:NNOMP} below, a variant of the Matching Pursuit algorithm~\cite{matchingpursuit} adapted to the nonnegative requirement.

\begin{algorithm}[h!]
\caption{Nonnegative Orthogonal Matching Pursuit.}
\label{algo:NNOMP}
\textbf{Input:} $J$, $b$, tolerance $\epsilon_{\rm Op. comp.}$\\
\textbf{Initialization:} $\mathcal{Z}=\emptyset$, $k'=0$, $\hat{\omega}=0$ and $r_0=b$
\begin{algorithmic}[1]
\While {$\left\|r_{k'}\right\|_2>\epsilon\left\|b\right\|_2$}
  \State $\mathcal{Z}\leftarrow \mathcal{Z}\cup {\rm max~index} \left(J_{[1:N_G]}^T r_{k'}\right)$
  \State $\hat{\omega}\leftarrow \underset{\hat{\omega}'>0}{\arg}\min\left\|b-J_\mathcal{Z}\hat{\omega}'\right\|_2$
  \State $r_{{k'}+1}\leftarrow b-J_\mathcal{Z} \hat{\omega}$
  \State ${k'}\leftarrow {k'}+1$
\EndWhile
\State $d\leftarrow {k'}$
\State $\hat{x}_k:=x_{\mathcal{Z}_k}$, $1\leq k\leq d$
\end{algorithmic}
\textbf{Output:} $\hat{\omega}_k$, $\hat{x}_k$, $1\leq k\leq d$.
\end{algorithm}

In this work, we consider the simple heuristics for the operator compression, introduced in~\cite[Section 3.3]{ECSW2}, and well adapted to a distributed context:
consider a subdomain $\Omega_l$,  $1\leq l\leq n_d$, and apply the Nonnegative Orthogonal Matching Pursuit (Algorithm~\ref{algo:NNOMP}) to derive a reduced quadrature formula to approximate the integration of the reduced internal forces vector locally on $\Omega_l$:
\begin{equation}
\sum_{k'=1}^{d^l}\hat{\omega}^l_{k'}\sigma\left(\epsilon(\hat{u}),y\right)(\hat{x}_{k'}^l,t):\epsilon\left(\psi_i\right)(\hat{x}_{k'}^l)\approx \int_{\Omega_l}\sigma\left(\epsilon(\hat{u}),y\right):\epsilon\left(\psi_i\right),
\end{equation}
with $1\leq i\leq n$, $\hat{\omega}^l_{k'} >0$, $\hat{x}^l_{k'}\in\Omega_l$, $1\leq k'\leq d^l$, $d^l\in\mathbb{N}$. Then, an approximation of the integration over the complete domain $\Omega$ can be obtained by simply summing the independently constructed local reduced quadrature formulas, where we recall that $\Omega_l\cap \Omega_{l'}=\emptyset$ for all $l\neq l'$, and $\displaystyle\bigcup_{l=1}^{n_d}\Omega_l=\Omega$. This procedure is suboptimal in the sense that for a given accuracy level, it will probably choose more integration points for the reduced quadrature than considering the complete domain.
However, it is executed independently in each subdomain without any communication between the processes.

\begin{remark}[(Discrete) Empirical Interpolation Method - (D)EIM]
(D)EIM is an \textit{offline}-\textit{online} procedure that computes an approximation for a separated variable representation of two-variable functions. The \textit{offline} stage contains a greedy procedure that selects, at each step, the values for each variable that maximizes the error of the currently constructed approximation, see~\cite{EIM1, EIM2} for details on the procedure.
Consider the computation of the reduced internal forces vector 
$\displaystyle\hat{F}^{\rm int}_i(t):=\int_\Omega\sigma\left(\epsilon(\hat{u}),y\right)(x,t):\epsilon\left(\psi_i\right)(x)$, $1\leq i\leq n$. Consider $\sigma\left(\epsilon(\hat{u}),y\right)$ as a function of time $t$ and space $x$, the EIM enables to recover an approximation in the form $\displaystyle\sigma\left(\epsilon(\hat{u}),y\right)(x,t)\approx\sum_{k=1}^d\alpha_k(t)\sigma\left(\epsilon(u),y\right)(x,t_l)$, where $\alpha_k(t)$ is the solution to $\displaystyle\sum_{l=1}^dB_{k,l}\alpha_l(t)=\sigma\left(\epsilon(\hat{u}),y\right)(x,t_k)$, with $B\in\mathbb{R}^{d\times d}$ a matrix constructed during the \textit{offline} stage of the EIM, and where we recall that $u$ is the solution to the high-fidelity model, and $\hat{u}$ is the reduced solution under construction. The sets $(x_l)_{1\leq l\leq d}$ and $(t_l)_{1\leq l\leq d}$, are referred to as magic points. Making use of the EIM approximation, the efficiency of the \textit{online} stage is recovered using $\displaystyle\hat{F}^{\rm int}_i(t)\approx\sum_{k=1}^d \alpha_k(t)\int_\Omega\sigma\left(\epsilon(u),y\right)(x,t_l):\epsilon\left(\psi_i\right)(x)$, where the integrals involve the high-fidelity solution $u$ at fixed time steps $t_l$ and can therefore be precomputed. Notice that for clarity of the presentation, we considered $\sigma$ as a scalar quantity, but the generalization to a $3\times 3$ order-2 tensor can be made by adding all the components in the spatial sampling.
While the EIM constructs an approximation of the integrand $\sigma\left(\epsilon(\hat{u}),y\right):\epsilon\left(\psi_i\right)$, the Discretized Empirical Interpolation Method (DEIM) constructs an approximation of the integrated vectors $\displaystyle\hat{F}^{\rm int}_i$ directly: the method has been proposed in~\cite{DEIM} and applied for nonlinear elasticity in~\cite{DEIM2}.
(D)EIM leads to a reduced problem that is not of Galerkin type, and no stabilization is guaranteed. In practice in our simulations, we did not manage to obtain stable POD-EIM schemes for elastoviscoplasticity, except when we stopped the greedy selection of the \textit{offline} stage of the EIM very early and when we kept only the first iteration of the reduced Newton algorithm~\eqref{eq:reducedNewton}. A stabilization procedure has been proposed in~\cite{GEIM1}.
\end{remark}

\begin{remark}[Selection of the reduced integration points: robustness and sparsity]
\label{randomrmk}
As stated before, the derivation of the reduced quadrature is a well-known but difficult problem. From~\cite[Section 5.1]{an2008optimizing}, the positivity constraint on the weight has an additional advantage: it protects the learning procedure against over-fitting: allowing the weights to take negative values, we can still continue to add integration points to increase the accuracy of the reduced quadrature on the learning data. However, doing so would increase the complexity of the quadrature formula, and eventually deteriorate the numerical stability of the quadrature (i.e. its ability to correctly integrate new data, outside the learning data). With the positivity constraint of the weight, we cannot reach any accuracy for the quadrature formula, but a low rank reconstruction is implicitly imposed, leading to sparse quadrature formula.
Consider line 4 of Algorithm~\ref{algo:NNOMP}: this is a linear regression model in absence of noise where the number of unknowns is (potentially much) larger that the number of observation. From~\cite{slawski}, in this precise context, the nonnegative constraints alone are sufficient to lead to a sparse solution, and the performance of the nonnegative least square with regard to prediction and estimation is comparable to that of the lasso. From this observation, we tried a simple heuristics for computing a reduced quadrature formula: taking random sets of candidate reduced integration points, and associated reduced integration weights solving line 4 of Algorithm~\ref{algo:NNOMP}, we can choose the one that minimizing the integration error. The sparsity is automatically obtained, and the performance of the procedure was globally comparable with the Nonnegative Orthogonal Matching Pursuit -- sometimes much more accurate solution were obtained. However, since it was not clear how to choose the size of the random sets and the performance changed from one execution to the other, we kept the Nonnegative Orthogonal Matching Pursuit for our numerical applications.
\end{remark}

\subsection{\textit{Online} stage and posttreatment}

The \textit{online} stage consists in applying the Galerkin method on the POD basis and treating the nonlinearity using a reduced Newton algorithm (see Section~\ref{sec:redprob}), where the reduced quadrature formulas are used to compute the integrals. In our implementation, the \textit{online} stage can be executed in sequential or in parallel, with or without the same number of processors as the \textit{offline} stage processing.

Depending on the application, the dual quantities (stress and constitutive law internal variables) may need to be reconstructed on the high-fidelity mesh. 
We propose a procedure using DEIM and Gappy-POD tools locally on each subdomains, described in Algorithm~\ref{algo:GappyPOD}, see~\cite{Everson95} for the original presentation of the Gappy-POD.

\begin{algorithm}[h!]
	\caption{Dual quantity reconstruction (DEIM and Gappy-POD) of the cumulated plasticity $p$.}
	\label{algo:GappyPOD}
          \textbf{\textit{Offline} stage:} consider subdomain $\Omega_l$, $1\leq l\leq n_d$
	\begin{algorithmic}[1]
	\State{Choose a tolerance $\epsilon_{\rm Gappy-POD}$}
	\State{Apply the snapshot POD (Algorithm~\ref{algo:snapshotPOD}) on the high-fidelity snapshots $p^l_s$, $1\leq s\leq N_c$ to obtain the orthonormal vectors $\psi^{p^l}_i$, $1\leq i\leq n^l$}
          \State{Apply the EIM to the collection of vectors $\psi^{p^l}_i$, $1\leq i\leq n^l$ to select $n^l$ distinct indices $\{j^l_k\}_{1\leq k\leq n^l}$ ($1\leq j^l_k\leq N_G$), and complete (without repeat) this set of indices by the indices of the integration points of the local reduced quadrature formula to obtain the set $\{j^l_k\}_{1\leq k\leq m^l}$ ($1\leq j^l_k\leq N_G$), where $m^l\leq n^l+d^l$}
          \State{Construct the matrix $M^l\in\mathbb{N}^{n^l\times n^l}$ such that $M^l_{i,j}=\sum_{k=1}^{m^l}\psi^{p^l}_i(x_{j^l_k})\psi^{p^l}_j(x_{j^l_k})$ (Gappy scalar product of the POD modes)}
	\end{algorithmic}
	\textbf{\textit{Online} stage:}  consider a time $t$ and a subdomain $l$, $1\leq l\leq n_d$
	\begin{algorithmic}[1]
	\State {Construct $b\in\mathbb{R}^{m^l}$, where $b_k$ is the \textit{online} prediction of $p^l$ at time $t$ and the indices $j^l_k$ (from the evaluation of the constitutive law solver at the corresponding integration point)}
	 \State{Solve the (small) linear system: $Mw=b$}
           \State {Compute the reconstructed value for $p^l$ on the complete subdomain $\Omega^l$ with $\sum_{i=1}^{n^l} w_i\psi^{p^l}_i$}
	\end{algorithmic}
\end{algorithm}

We now give some comments on Algorithm~\ref{algo:GappyPOD}.
Recall that during the \textit{online} resolution of the reduced order model, the constitutive law is solved for a set of reduced integration points selected during the operator compression step. In practice in our case, applying the Gappy-POD using the information at the reduced integration points did not provide satisfying reconstructions at the extrapolated cycles for the cumulated plasticity field, which is of prime interest for lifetime computations. It has been noted in~\cite{willcox2006} that the mask used for the Gappy-POD (in our case, the set of reduced integration points) plays an important role in the quality of the reconstruction, namely one has to ensure that the matrix $M$ is invertible. We now explain how our choices in Algorithm~\ref{algo:GappyPOD} ensure the invertibility of the matrix $M$.
\begin{proposition}[Well-posedness of the reconstruction]
The linear system considered in the \textit{online} stage of Algorithm~\ref{algo:GappyPOD} is invertible.
\end{proposition}
\begin{proof}
The EIM applied to the collection of vectors $\psi^{p^l}_i$, $1\leq i\leq n^l$, produced the set of integration points $j^l_k$, $1\leq k\leq n^l$, and a permutation $\zeta^l$ of  $\{1, 2, \cdots, n^l\}$, since we allow the greedy algorithm in the EIM to compute $n^l$ iterations. Notice that this is possible since the vectors $\psi^{p^l}_i$, $1\leq i\leq n^l$ is a free family. 
Define now the matrix $F\in\mathbb{R}^{n^l\times n^l}$ such that $F_{i,k}=\psi^{p^l}_{\zeta^l(i)}(x_{j^l_k})$. From~\cite[Theorem 1.2]{casenave2016}, $F$ is invertible. In particular, the matrix $\left(\psi^{p^l}_{i}(x_{j^l_k})\right)_{1\leq i\leq n^l,1\leq k\leq m^l}$ (where we consider the indices of the integration points as well as the indices selected by EIM) is of rank $n^l$, and $M$ is then invertible.
\end{proof}

In Algorithm~\ref{algo:GappyPOD}, the construction of the POD basis and the EIM algorithm selecting $n$ indices represents the DEIM part. In practice, we used a recently proposed implementation: the QDEIM~\cite{QDEIM}, which consists in a QR decomposition with column pivoting applied to the collection of the POD modes $\psi^{p^l}_i$, $1\leq i\leq n^l$. The part of Algorithm~\ref{algo:GappyPOD} consisting in solving the (small) linear system and recombining the solution with the POD modes $\psi^{p^l}_i$, $1\leq i\leq n^l$ corresponds to the Gappy-POD treatment, for which we have ensured well-posedness. The DEIM part enables to add, in the reduced integration points set, new points associated with the approximation of the cumulated plasticity field, which improved the quality of the reconstruction for this field.


\section{Numerical applications}
\label{sec:num}

In Section~\ref{sec:ansys} is presented an application using the commercial code Ansys, to illustrate the nonintrusivity capabilities of our framework.
In Section~\ref{sec:evp}, we consider our main application of interest: the cyclic extrapolation of an elastoviscoplastic structure using a reduced order model. 
In what follows, ROM means Reduced Order Model, and HFM means High-Fidelity Model.

Two constitutive laws are considered in the numerical applications:
\begin{itemize}
\item[(elas)] temperature-dependent cubic elasticity and isotropic thermal expansion: for this law, there no internal variable to compute, and the constitutive law is $\sigma=\mathcal{A}:\left(\epsilon-\epsilon^{\rm th}\right)$, where $\epsilon^{\rm th}=\alpha\left(T-T_0\right)I$, with $I$ the second-order identity tensor and $\alpha$ the thermal expansion coefficient in MPa.K${}^{-1}$ depending on the temperature, and where $\mathcal{A}$, the elastic stiffness tensor, does not depend on the solution $u$ and is defined in Voigt notations by
\begin{equation}
\label{eq:tensorElas}
\mathcal{A}=
\begin{pmatrix} 
y_{1111} & y_{1122} & y_{1122} & 0 & 0 & 0 \\
y_{1122} & y_{1111} & y_{1122} & 0 & 0 & 0 \\
y_{1122} & y_{1122} & y_{1111} & 0 & 0 & 0 \\
0&0 &0 & y_{1212} & 0&0 \\
0&0 & 0 &0 &  y_{1212} &0 \\
0&0  &0 &0 &0 & y_{1212}  \\
\end{pmatrix},
\end{equation}
where the temperature $T$ is given by the thermal loading, $T_0=20^\circ$C is a reference temperature and the coefficients $y_{1111}$, $y_{1122}$ and $y_{1212}$ (Young moduli in MPa) depend on the temperature.

\item[(evp)] Norton flow with nonlinear kinematic hardening and constant isotropic hardening: for this law, the elastic part is given by $\sigma=\mathcal{A}:\left(\epsilon-\epsilon^{\rm th}-\epsilon^P\right)$, where $\mathcal{A}$ and $\epsilon^{\rm th}$ are the same as the (elas) law, $\epsilon^P$ is the plastic strain tensor, and the viscoplastic part requires solving the following ODE: $\dot{\epsilon}^P=\dot{p}\left(n_r-D\epsilon^P\right)$, where $p$ is the cumulated plasticity such that $\dot{p}=\left\langle\frac{f_r}{K}\right\rangle^m$, with yield criterion $f_r:=s_{\rm eq}-R_0\geq 0$, where $s:=\left(\sigma-\frac{1}{3}{\rm Tr}(\sigma)I\right)-C\epsilon^P$. In what precedes, $n_r:=\frac{3}{2}\frac{s}{s_{\rm eq}}$ is the normal to the yield surface (with $\left\|\cdot\right\|$ the Euclidean norm and Tr the trace operator), $\langle\cdot\rangle$ denotes the postive part operator, $\cdot_{\rm eq}=\sqrt{\frac{3}{2}}\left\|\cdot\right\|$.
The coefficients $C$ and $D$ are hardening material coefficients (in MPa), $K$ is the Norton material coefficient (in MPa), $m$ is the Norton exponential material coefficient (dimensionless) and $R_0$ is the isotropic hardening material coefficient (in MPa): they all depend on the temperature.
The internal variables considered here are $\epsilon^P$ and $p$, and the ODE is completed with the initial conditions $\epsilon^P=0$ and $p=0$ at $t=0$.
\end{itemize}


\subsection{Rotating bike crank using Ansys Mechanical}
\label{sec:ansys}

\setlength\figureheight{6cm}
\setlength\figurewidth{6cm}
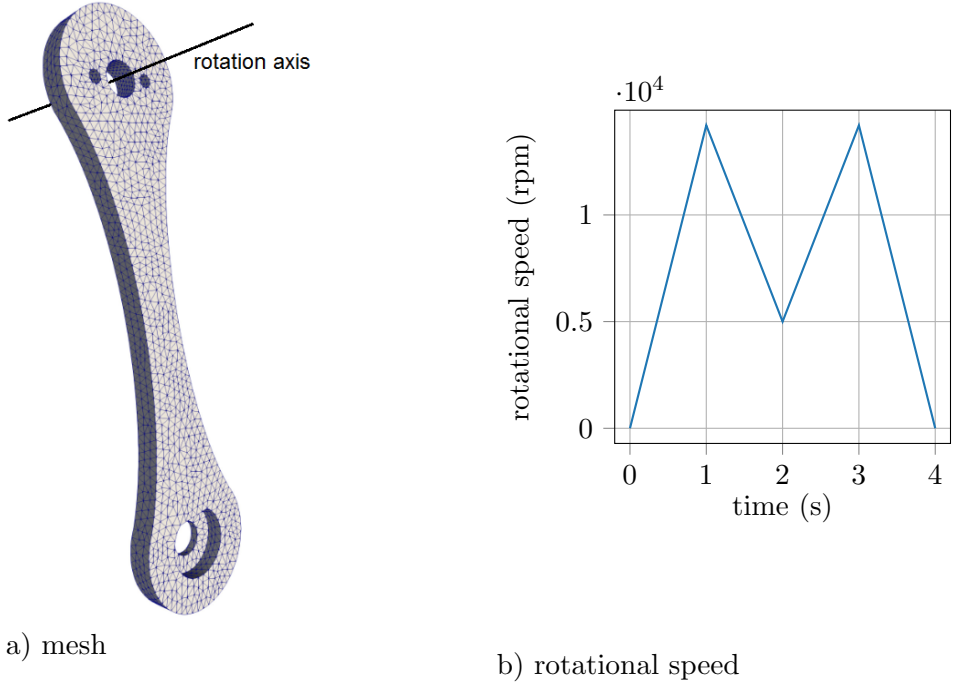
\begin{figure}[h]
  \centering
  \begin{minipage}{.24\linewidth}
   \end{minipage} \hfill
  \begin{minipage}{.24\linewidth}
 \includegraphics[width=\textwidth]{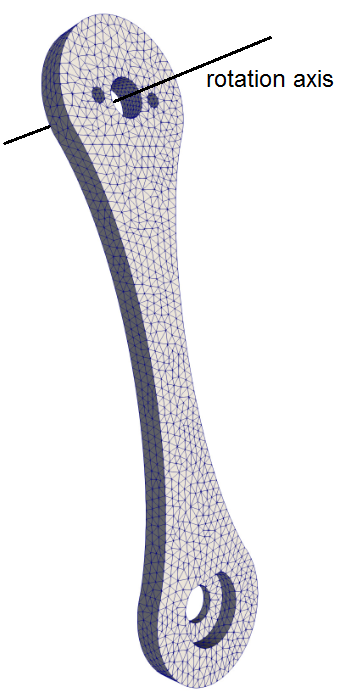}\\
   a) mesh
   \end{minipage} \hfill
   \begin{minipage}{.48\linewidth}
\vspace{1.4cm}
\begin{tikzpicture}

\definecolor{color0}{rgb}{0.12156862745098,0.466666666666667,0.705882352941177}

\begin{axis}[
xlabel={time (s)},
ylabel={rotational speed (rpm)},
xmin=-0.2, xmax=4.2,
ymin=-710, ymax=14910,
width=\figurewidth,
height=\figureheight,
tick align=outside,
tick pos=left,
xmajorgrids,
x grid style={white!69.019607843137251!black},
ymajorgrids,
y grid style={white!69.019607843137251!black}
]
\addplot [thick, color0, forget plot]
table {%
0 0
1 14200
2 5000
3 14200
4 0
};
\end{axis}

\end{tikzpicture}
\vspace{1.4cm}\\
   b) rotational speed
   \end{minipage}
  \caption{a) Mesh for the bike crank, b) Rotational speed with respect to time.}
  \label{fig:crankloading}
\end{figure}

To illustrate the nonintrusivity capabilities of our framework, we consider a simple test case, where we only assess the ability of the reduced order model to recompute the same loading as the high-fidelity computation used in the data generation step. The structure is maintained at a temperature of 20\textdegree C and is subject to very strong rotational effects, see Figure~\ref{fig:crankloading} for a representation of the mesh and a description of the loading. The loading is not physical: it has been chosen to trigger the strongest possible strain-induced elastoviscoplastic effects: a slightly stronger rotation makes Ansys fail to compute the result, as we reach the limit of the strength of the material. Doing so, we challenge the ability of the framework to deal with a strongly nonlinear problem.

\begin{table}
  \centering
\begin{tabular}{|c|c|}
  \hline
  number of dofs & 101'952 \\
  \hline
  number of (quadratic) tetrahedra & 21'409 \\
  \hline
  number of integration points & 107'045 \\
  \hline
  number of time steps & 30 per cycle \\
  \hline
 constitutive law &\begin{tabular}{@{}c@{}}evp (Norton flow with nonlinear kinematic hardening\\ and constant isotropic hardening)\end{tabular} \\
  \hline
\end{tabular}
  \caption{Characteristics for the bike crank test case.}
  \label{tab:crankcase}
\end{table}

The characteristics of the test-case are provided in Table~\ref{tab:crankcase}.
Ansys only provides the dual-quantities on the nodes of the mesh, by interpolating them from their values at the integration points. To access the high-fidelity quadrature scheme (in order to compute an accurate reduced quadrature in the operator compression step), we 
recompute the evolution of the dual-quantities at the integration points using the constitutive law solver and the global equilibrium at each time step given by Ansys and the finite element interpolation by our framework, before the data compression step.
This test case is computed in sequential for the data generation, data compression and operator compression steps. The \textit{online} stage is computed in sequential as well.

\begin{figure}[h]
  \centering
  \includegraphics[width=0.8\textwidth]{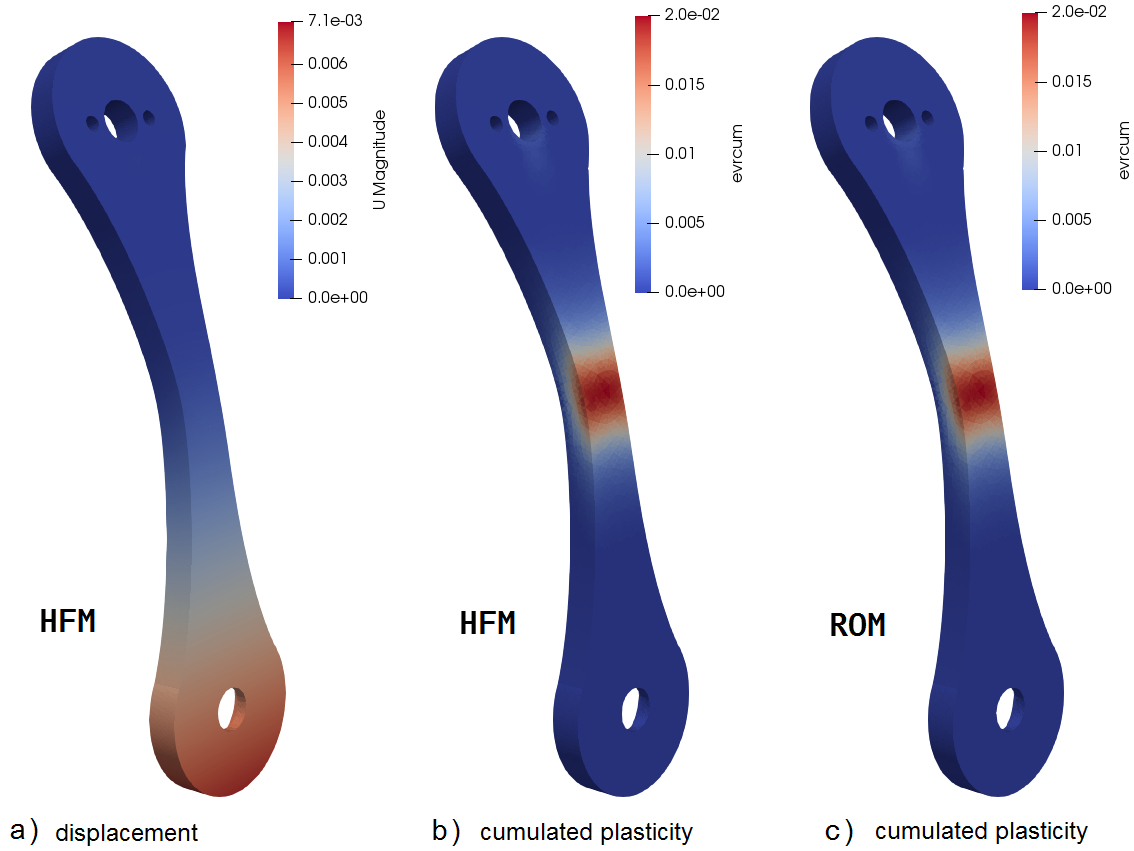}
  \caption{Solution fields on the deformed mesh (amplified 5 times) at $t=2.2$s: a) displacement for the HFM, b) cumulated plasticity for the HFM, c) cumulated plasticity for the ROM.}
  \label{fig:crank}
\end{figure}

In Figure~\ref{fig:crank} are represented some solution fields at time $t=2.2$s. We notice that the reduced order model has correctly reconstructed the plasticity field, then illustrating the ability of the framework to reduced strongly nonlinear problems computed by commercial software, i.e. in a nonintrusive fashion.

We now present a meaningful industrial test case with a physical loading and strain and temperature-induced elastoviscoplastic effects, and such that the reduced model is not limited to reproduce only the high-fidelity computation.

\subsection{Elastoviscoplastic cycle extrapolation of a high-pressure turbine blade using Zset}
\label{sec:evp}

In this section, we use the presented framework for the application of industrial interest motivated in the introduction, namely the efficient computation of a large number of cycles of a high-pressure turbine blade undergoing a loading representing aircraft missions (either take-off-cruise-landing or test-bench cycles). The lifetime prediction being a rapidly computed post-treatment of the cyclic simulation, we focus our task to the efficient computation of many cycles.

As numerical procedure, we propose to compute the first cycle using a high-fidelity simulation, then construct a reduce-order model to compute a large number of cycles. Notice that this is not the classical context of reduced order modeling,  since the ROM is used for the acceleration of a single computation. We need efficient computations for the complete numerical sequence, since, unlike many-queries contexts, long run-times for the \textit{offline} stage cannot be compensated by an intensive use of the ROM. As a consequence, we give the following definition for the speedup:
\begin{definition}
The speedup is defined as the ratio between the run-time for the high-fidelity code and the run time of the complete numerical procedure, including the three steps of the \textit{offline} stage: data generation (first cycle high-fidelity cycle), data compression and operator compression, as well as the \textit{online} stage.
\end{definition}
With this definition of the speedup, we can evaluate the opportunity of using reduced order modeling tools to accelerate a single computation.

We consider a model of high-pressure turbine blade, with 4 internal cooling channels. The lower part, called the foot, is modeled by an elastic material (even if plasticity occurs here, it is not of interest for computing the lifetime of the blade) whereas the upper part is modeled by an elastoviscoplastic law. The high-fidelity computation is carried-out in parallel on 48 processors using Zset~\cite{zset} with an AMPFETI solver (Adaptive MultiPreconditioned FETI~\cite{bovet2017}), see Figure~\ref{fig:aube}.
Notice that in order to obtain a fast complete procedure, we also spend efforts to accelerate the computation of the high-fidelity model, which is not quantified by our definition of the speedup.
The loading is represented in Figure~\ref{fig:aubeloading}: the axis of rotation is directed along the x-axis (and located far below the blade) and the rotational speed consists in a linear increase followed by a linear decrease; the temperature field starts by a uniform 20\textdegree C at $t=0$s, is linearly interpolated node by node up to the temperature field represented in Figure~\ref{fig:aubeloading}~b) at $t=20$s and decreases again linearly up to a uniform 20\textdegree C at $t=40$s. The characteristics of the test-case are provided in Table~\ref{tab:bladecase}.

\begin{figure}[h]
  \centering
  \includegraphics[width=\textwidth]{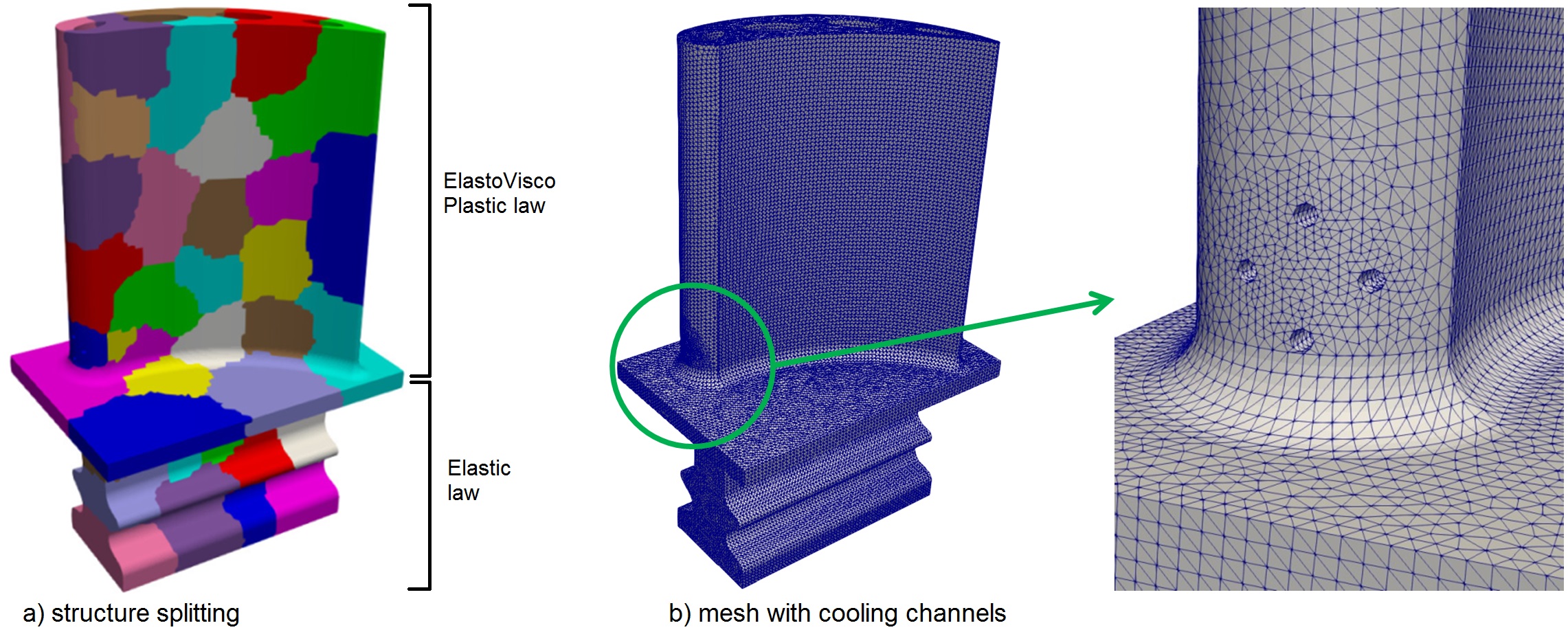}
  \caption{a) Structure split in 48 subdomains - the top part of the blade's material is modeled by an elastoviscoplastic law and the foot's one by an elastic law, b) Mesh for the high-pressure turbine blade with a zoom on the cooling channels.}
  \label{fig:aube}
\end{figure}

\setlength\figureheight{6cm}
\setlength\figurewidth{6cm}
\begin{figure}[h]
  \centering
  \begin{minipage}{.48\linewidth}
\vspace{1.2cm}
\begin{tikzpicture}

\definecolor{color0}{rgb}{0.12156862745098,0.466666666666667,0.705882352941177}

\begin{axis}[
xlabel={time (s)},
ylabel={rotational speed (rpm)},
xmin=-2, xmax=42,
ymin=-700, ymax=14700,
width=\figurewidth,
height=\figureheight,
tick align=outside,
tick pos=left,
xmajorgrids,
x grid style={white!69.019607843137251!black},
ymajorgrids,
y grid style={white!69.019607843137251!black}
]
\addplot [thick, color0, forget plot]
table {%
0 0
20 14000
40 0
};

\end{axis}

\end{tikzpicture}
\vspace{1.2cm}\\
   a) rotational speed
   \end{minipage} \hfill
  \begin{minipage}{.48\linewidth}
 \includegraphics[width=0.86\textwidth]{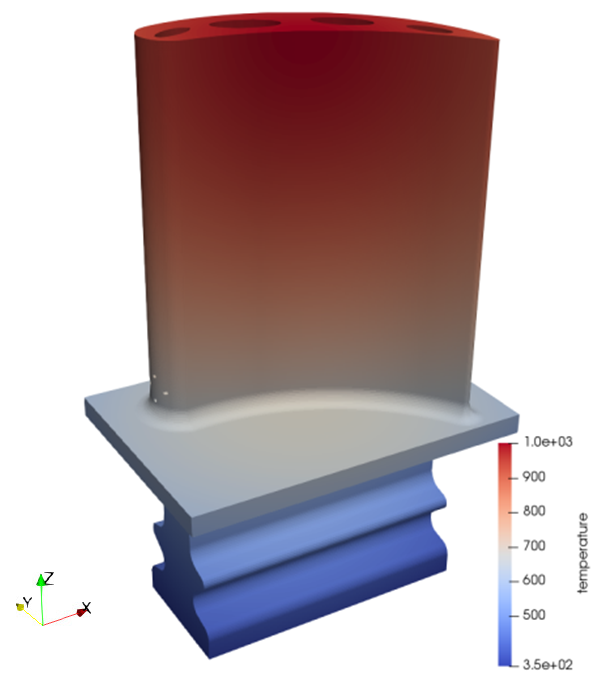}\\
   b) temperature field at max loading
   \end{minipage}
  \caption{a) Rotational speed with respect to time, b) Temperature field at maximum loading ($t=20$s).}
  \label{fig:aubeloading}
\end{figure}

\begin{table}
  \centering
\begin{tabular}{|c|c|}
  \hline
  number of dofs & 4,892'463 \\
  \hline
  number of (quadratic) tetrahedra & 1'136'732 \\
  \hline
  number of integration points & 5'683'660\\
  \hline
  number of time steps & 40 per cycle \\
  \hline
 constitutive law for the foot &\begin{tabular}{@{}c@{}}elas (temperature-dependent cubic elasticity\\and isotropic thermal expansion)\end{tabular} \\
  \hline
 constitutive law for the blade &\begin{tabular}{@{}c@{}}evp (Norton flow with nonlinear kinematic hardening\\ and constant isotropic hardening) \end{tabular} \\
  \hline
\end{tabular}
  \caption{Characteristics for the high-pressure turbine blade test case.}
  \label{tab:bladecase}
\end{table}

Details are provided in Table~\ref{tab:bladecomputation} on the computational procedure. In particular, the computations for all the steps are done in parallel with distributed memory and using MPI for the communications between processors (48 processors within 2 nodes). Notice that in our framework, even the visualization is done in parallel with distributed memory using a parallel version of Paraview~\cite{paraview1, paraview2}. The choices for the stopping criteria are made rather empirical $\epsilon^{\rm HFM}_{\rm Newton}$ and $\epsilon^{\rm ROM}_{\rm Newton}$ are chosen equal to $10^{-6}$, $\epsilon_{\rm POD}=10^{-7}$ and $\epsilon_{\rm Gappy-POD}=10^{-7}$ provide the more accurate possible compression (below this value for the data compression, the representation of the snapshots on the basis is not more accurate), $\epsilon_{\rm Op. comp.}=10^{-5}$: below this value, the convergence of the  Distributed NonNegative Orthogonal Matching Pursuit is very slow, and even never reached in some of our numerical tests, see Remark~\ref{randomrmk}. The complete procedure takes 6h31min to compute an approximation for the 100 cycles, whereas the high-fidelity model would have taken 9 days and 14 hours, which corresponds to a speedup of approximately 35. While this speedup may appear low, we recall that all the \textit{offline} stage is taken into account, and that we used a state-of-the-art AMPFETI solver for the high-fidelity reference. Using a standard commercial software for the high-fidelity reference would have lead to a more impressive speedup. With this test-case, we illustrate that a computation that takes many days even with state-of-the-art high-fidelity solvers can be done overnight, which is of great practical interest in design phases of a mechanical part in an industrial context.

\begin{table}
  \centering
{\tiny
\begin{tabular}{|c|c|c|}
  \hline
 step & algorithm &  wall-time (CPUs) \\
  \hline
 Data generation & AMPFETI solver in Zset, $\epsilon^{\rm HFM}_{\rm Newton}=10^{-6}$ & 2h18min~(48)\\
  \hline
 Data compression & Distributed Snapshot POD, $\epsilon_{\rm POD}=10^{-7}$ & 9min~(48)\\
  \hline
 Operator compression & Distributed NonNegative Orthogonal Matching Pursuit, $\epsilon_{\rm Op. comp.}=10^{-5}$ & 44min~(48)\\
  \hline
 \textit{Online} computation & ROM for 100 cycles, $\epsilon^{\rm ROM}_{\rm Newton}=10^{-6}$ & 3h13min~(48)\\
  \hline
 Reconstruction of $\sigma$, $\epsilon$ and $p$ & Distributed Gappy-POD, $\epsilon_{\rm Gappy-POD}=10^{-7}$ & 7min~(48)\\
  \hline
\end{tabular}
}
  \caption{Description of the computational procedure ($p$ denotes the cumulated plasticity).}
  \label{tab:bladecomputation}
\end{table}

The set of integration points selected for the operator compression step is represented in Figure~\ref{fig:rid}.
\begin{figure}[h]
  \centering
  \includegraphics[width=\textwidth]{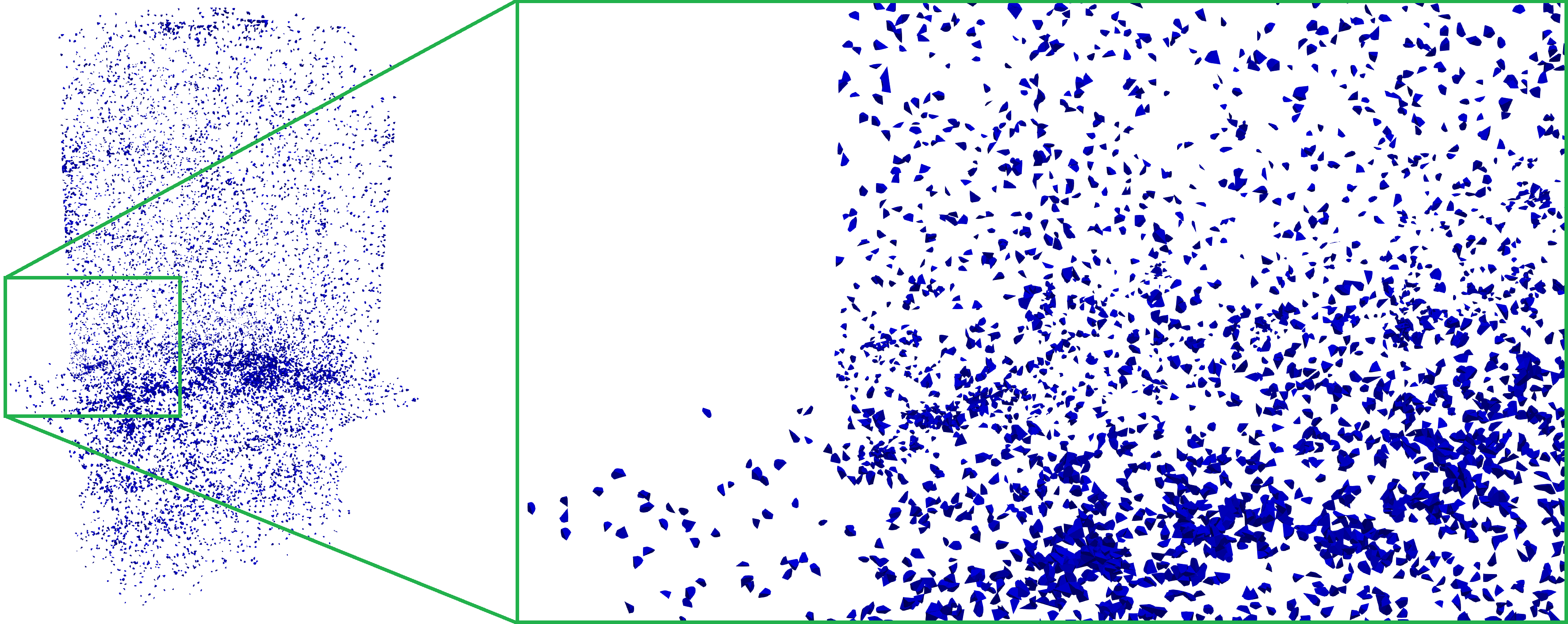}
  \caption{Set of integration points selected for the operator compression step: each integration point is materialized by a polyhedron obtained by Voronoi tesselation of the mesh, where the seeds are the integration points. Selecting all the integration points would lead to representing the complete structure. We notice smaller polyhedra around the internal cooling channels due to a refined mesh.}
  \label{fig:rid}
\end{figure}

To assess the accuracy of the ROM and its ability to extrapolate the cyclic loading, the HFM has been run on 10 cycles, for a duration of 23h20min. We recall that the ROM is computed from the information of the first cycle of the HFM.
To comply with a real-life situation, the HFM has not been computed up to the $100^{\rm th}$ cycle, the assessment of the extrapolation quality being made by comparing the solutions at the $10^{\rm th}$ cycle. In Figure~\ref{fig:res1}, we compare the displacement and the stress at the middle of the $1^{\rm st}$ and $10^{\rm th}$ cycles. For these quantities, the extrapolation works well, but the solution is not changing much between the $1^{\rm st}$ and $10^{\rm th}$ cycles.
In Figure~\ref{fig:res2}, are compared the displacement on the deformed mesh amplified $1000$ times at the end of the cycles, when the loading is released. The residual displacement is well predicted. In Figures~\ref{fig:res3} and~\ref{fig:res4} are represented the stress and cumulated plasticity fields at the end of the cycles. These quantity undergo a visible evolution during the cycles and are of prime interest for the lifetime predictions. By comparing the zones circled in green, we notice that the ROM correctly extrapolates the evolution of the fields between the $1^{\rm st}$ and $10^{\rm th}$ cycles. In the zones circled in red, we notice some discontinuities of the Gappy-POD reconstruction between the subdomains. This is due to the independent local treatments of the Gappy-POD procedure. The representation with saturated colors helps the visualization but amplifies the discontinuities for the small values. In Figure~\ref{fig:diffmaps} are represented difference maps between the ROM and the HFM for the displacement and the cumulated plasticity.
Finally, in Figure~\ref{fig:res5} are represented cyclic evolution of $\sigma_{33}$ with respect to $\epsilon_{33}$ and cumulated plasticity with respect to time, for two integration points. Although the first cycle appears very well approximated by the ROM, errors are accumulated in the following cycles.

\begin{figure}[h]
  \centering
  \includegraphics[width=\textwidth]{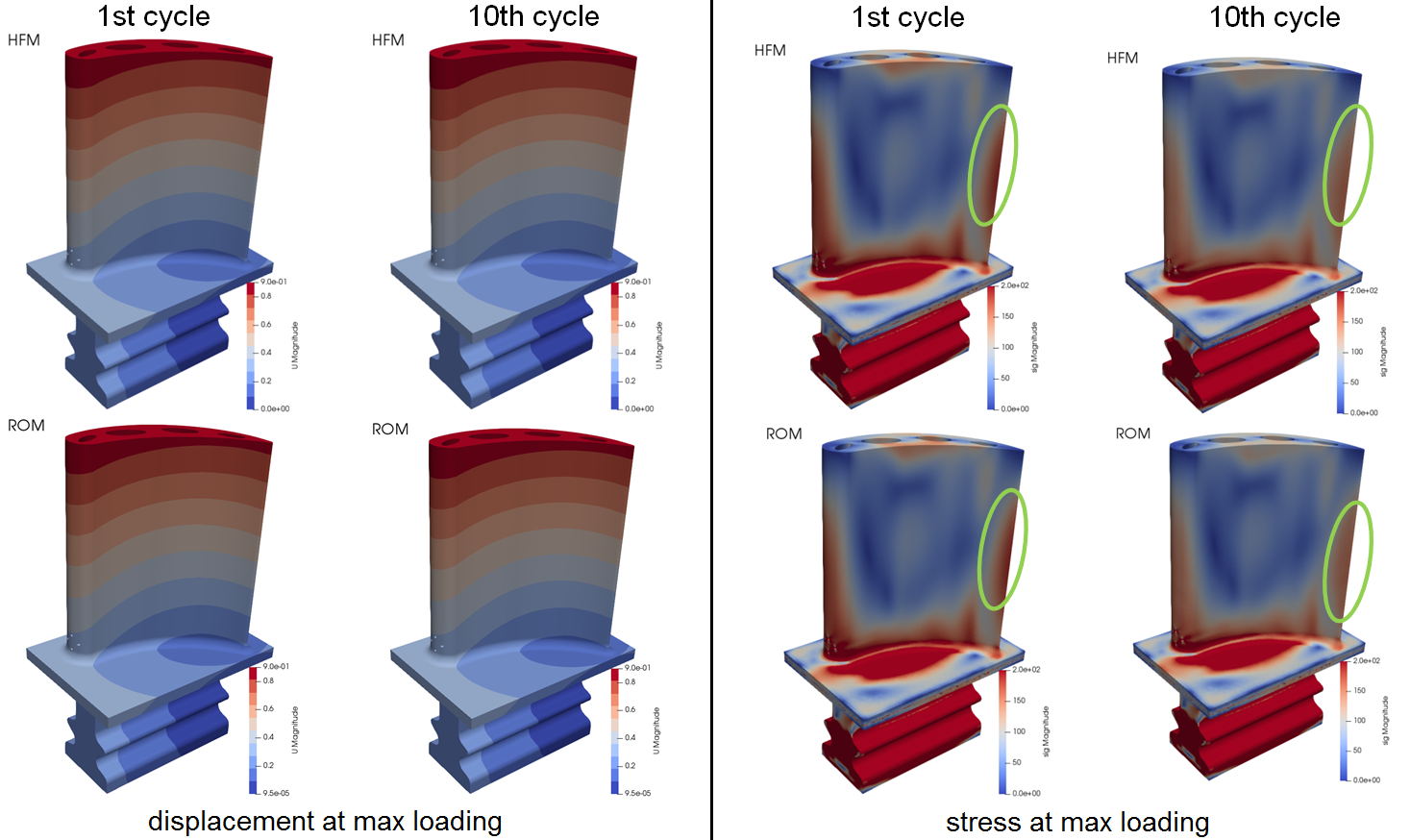}
  \caption{Comparison between HFM and ROM at the middle of the $1^{\rm st}$ and $10^{\rm th}$ cycles (maximum loading). Left-hand side: magnitude of the displacement, right-hand side: magnitude of the stress.}
  \label{fig:res1}
\end{figure}

\begin{figure}[h]
  \centering
  \includegraphics[width=\textwidth]{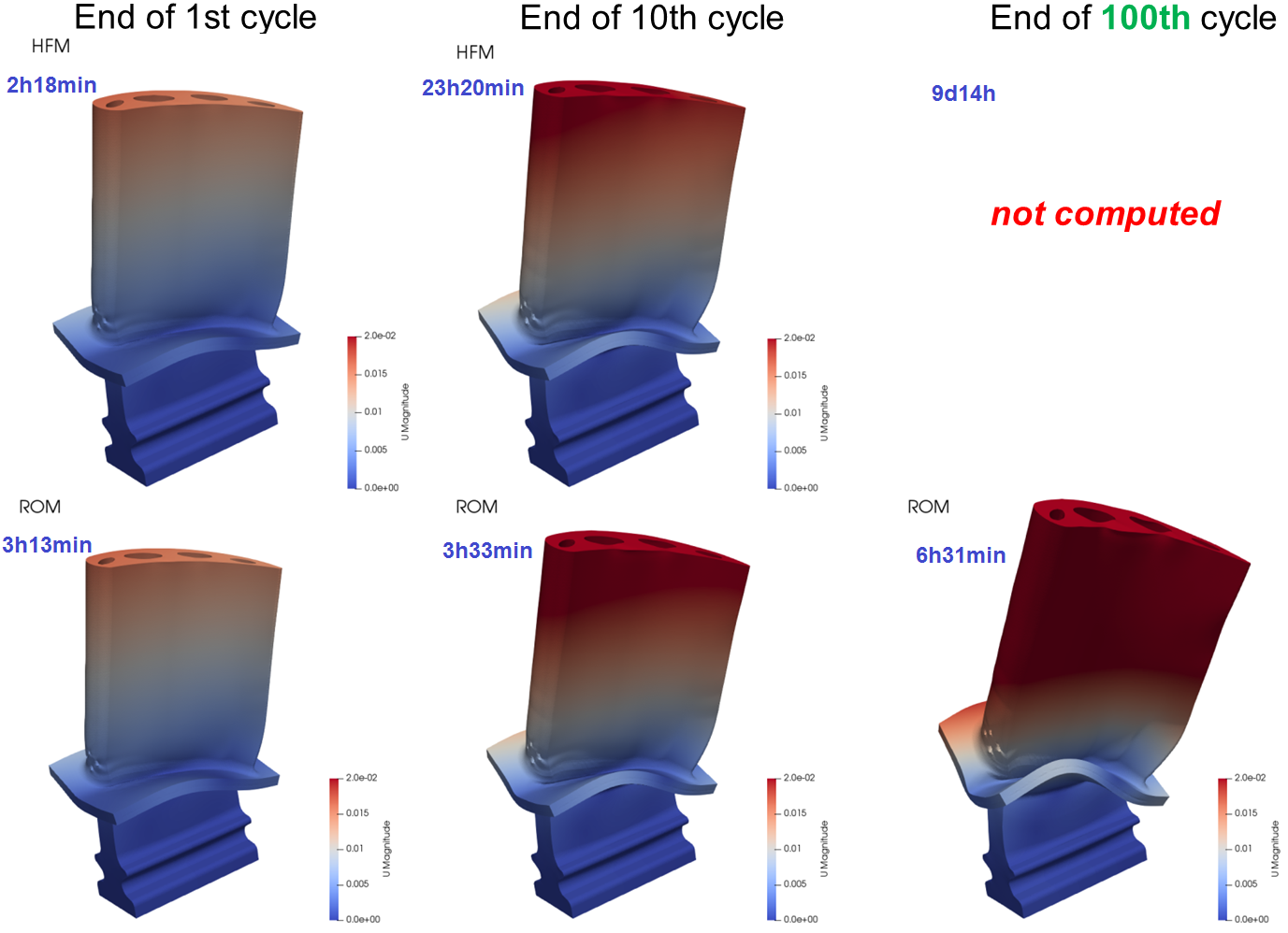}
  \caption{Magnitude of the displacement fields on a deformed mesh (amplified 1000 times) for the HFM and the ROM at the end of the $1^{\rm st}$ and $10^{\rm th}$ cycles, and  for the ROM at the end of the $100^{\rm th}$ cycle (minimum loading).}
  \label{fig:res2}
\end{figure}

\begin{figure}[h]
  \centering
  \includegraphics[width=\textwidth]{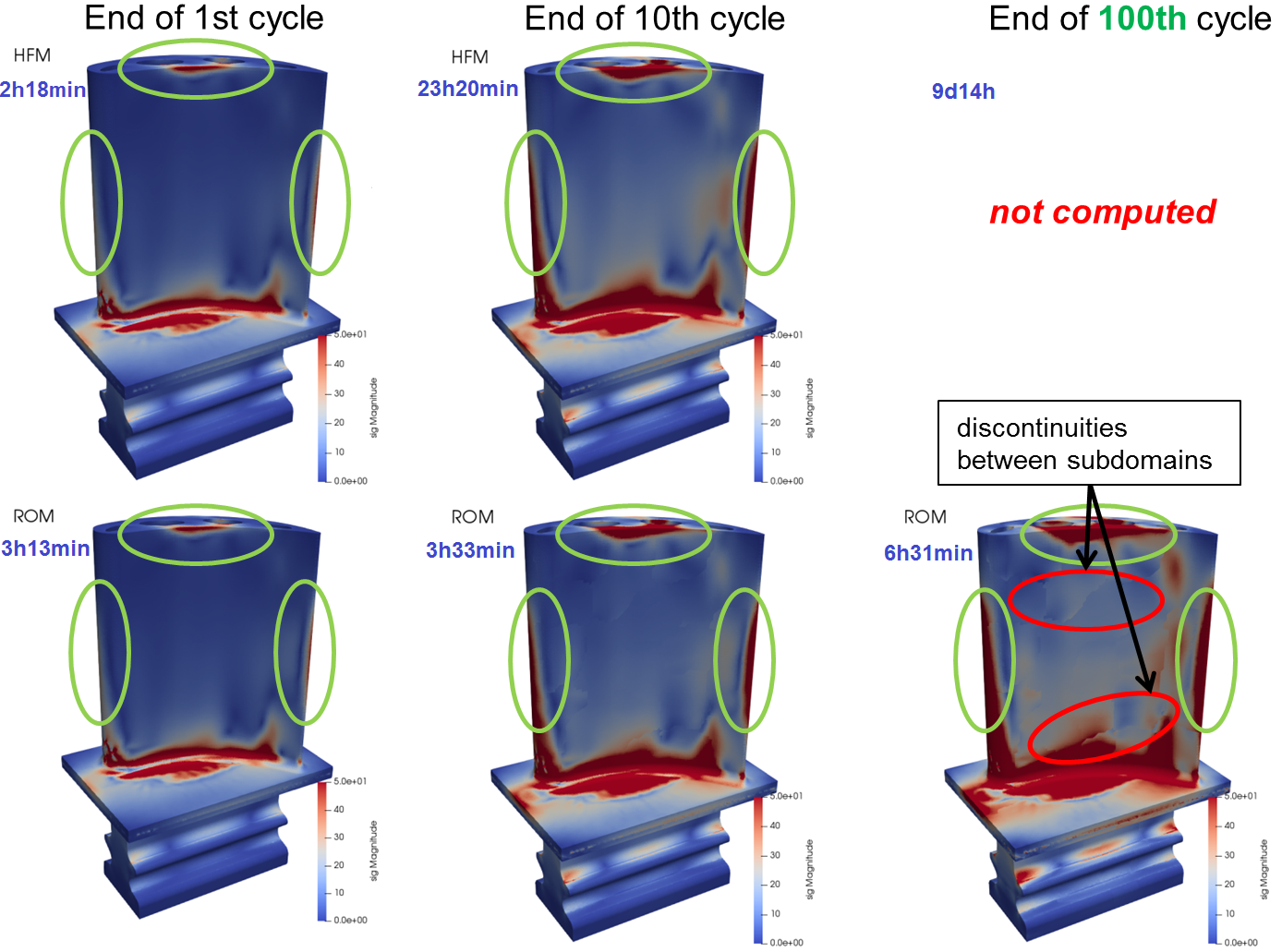}
  \caption{Magnitude of the stress fields for the HFM and the ROM at the end of the $1^{\rm st}$ and $10^{\rm th}$ cycles, and  for the ROM at the end of the $100^{\rm th}$ cycle (minimum loading) - with saturated colors.}
  \label{fig:res3}
\end{figure}

\begin{figure}[h]
  \centering
  \includegraphics[width=\textwidth]{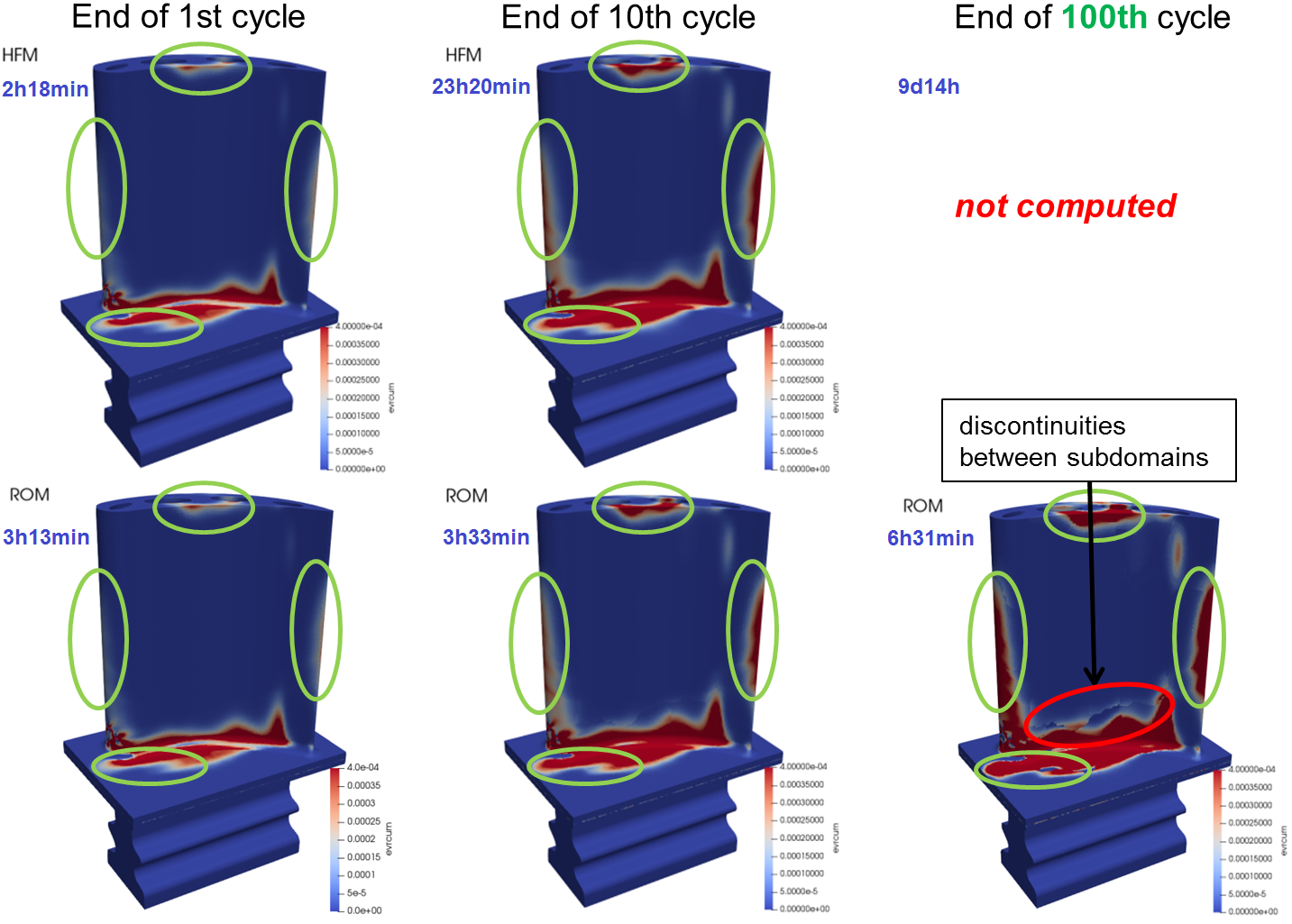}
  \caption{Magnitude of the cumulated plasticity fields for the HFM and the ROM at the end of the $1^{\rm st}$ and $10^{\rm th}$ cycles, and  for the ROM at the end of the $100^{\rm th}$ cycle (minimum loading) - with saturated colors.}
  \label{fig:res4}
\end{figure}

\begin{figure}[h]
  \centering
  \begin{minipage}{.245\linewidth}
 \includegraphics[width=0.936\textwidth]{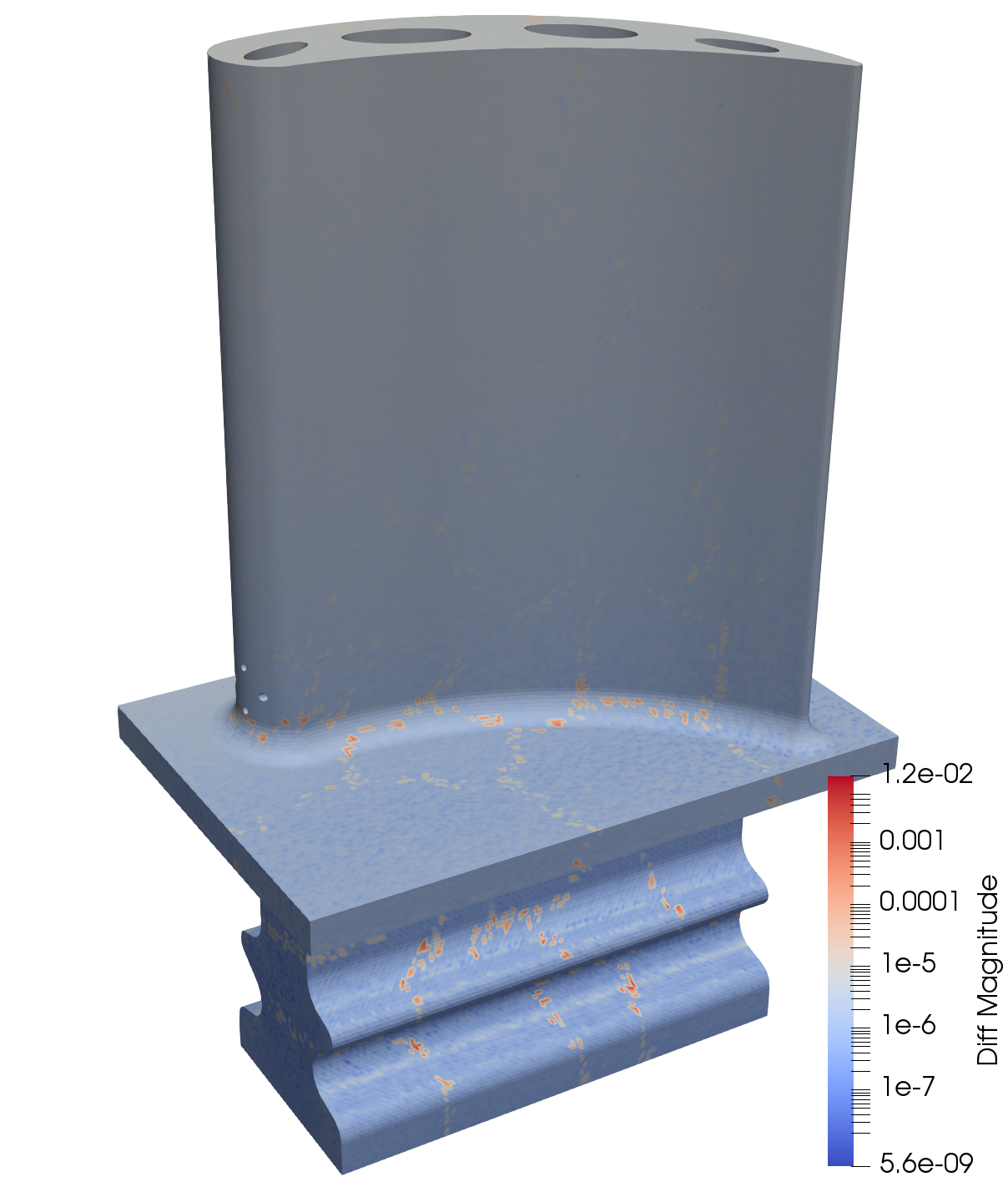}\\
   a) displ. error $1^{\rm st}$ cycle
   \end{minipage} \hfill
  \begin{minipage}{.245\linewidth}
 \includegraphics[width=0.936\textwidth]{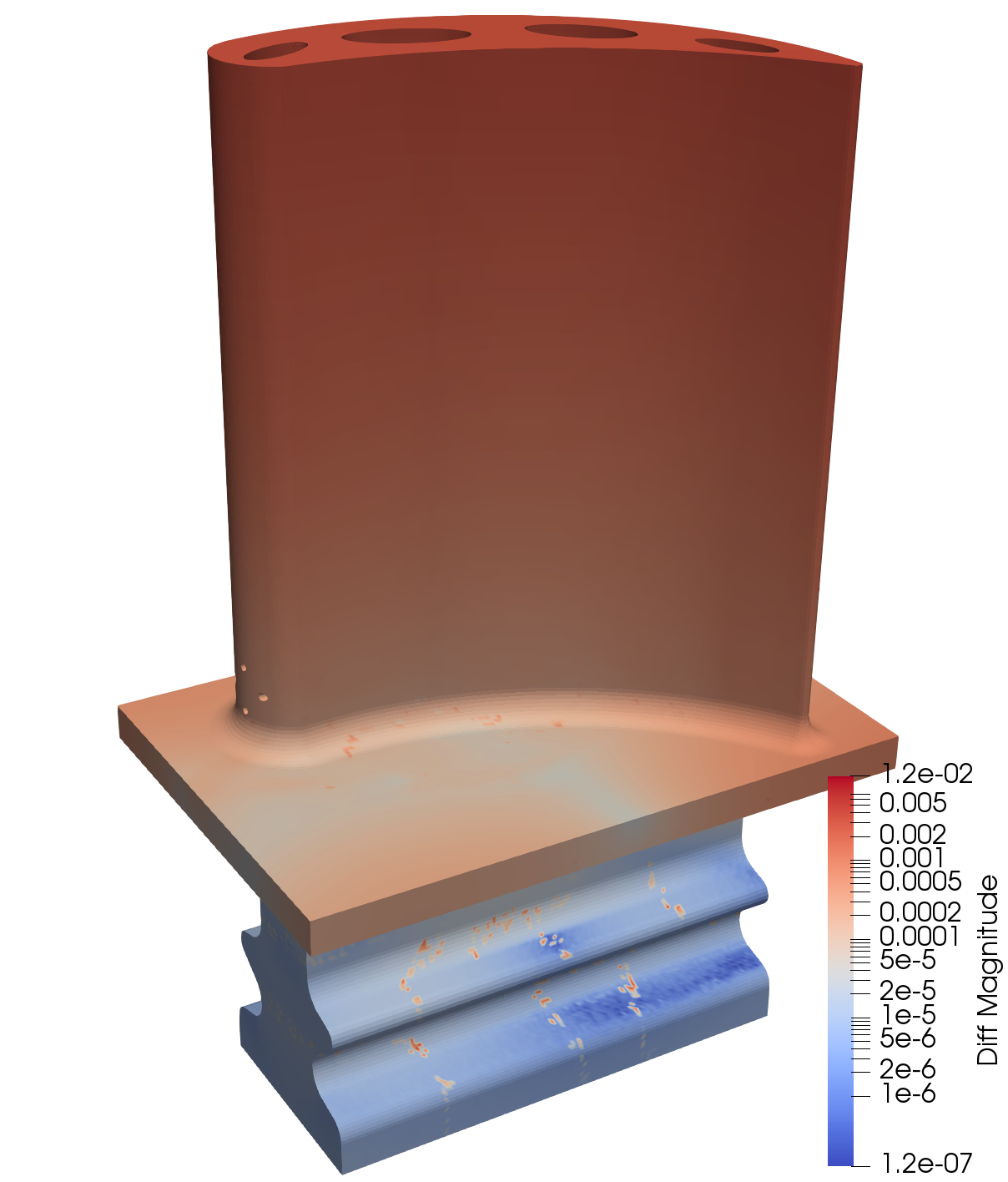}\\
   b) displ. error $10^{\rm th}$ cycle
   \end{minipage} \hfill
  \begin{minipage}{.245\linewidth}
 \includegraphics[width=0.984\textwidth]{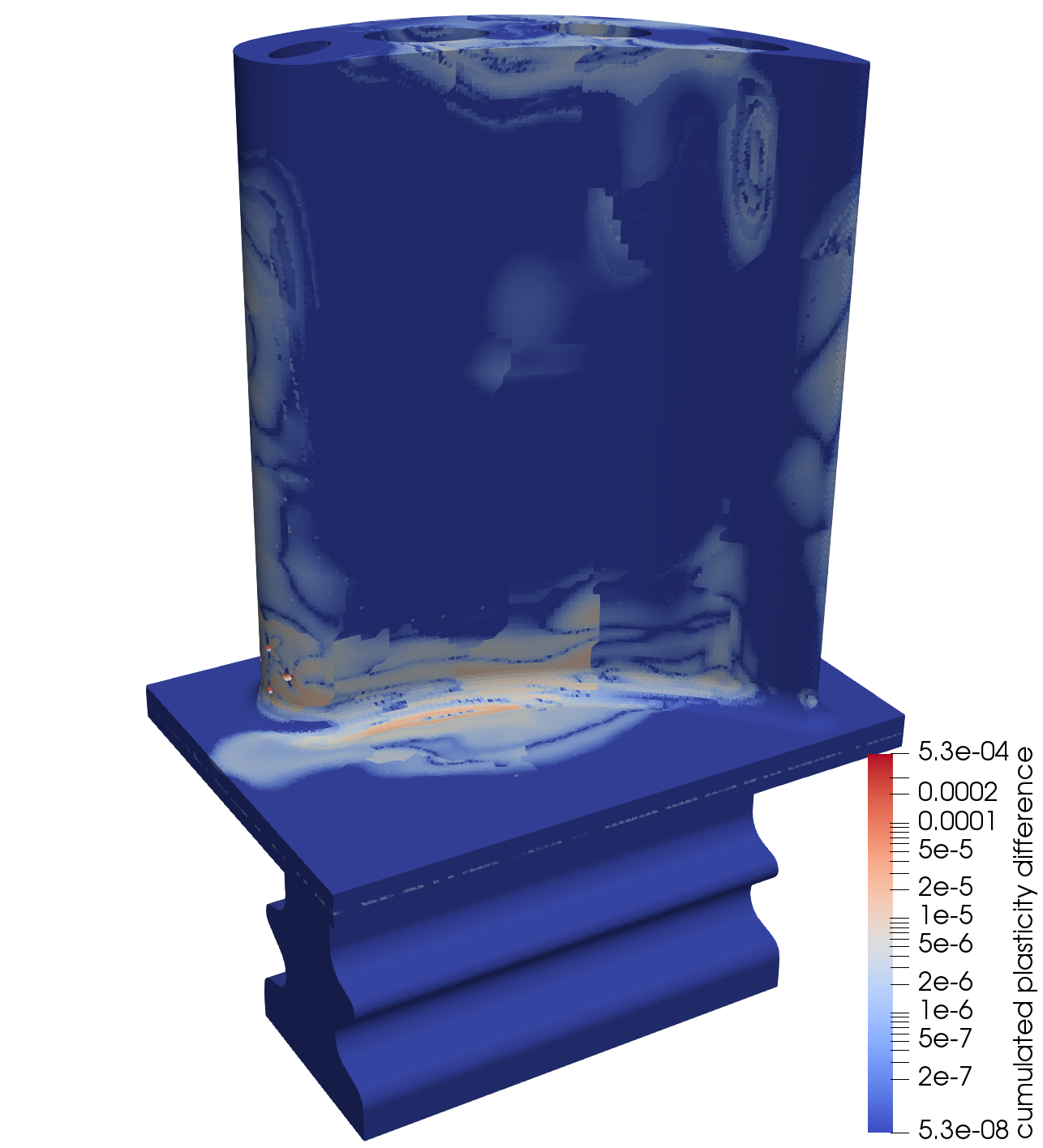}\\
   c) cumul. plast. error $1^{\rm st}$ cycle
   \end{minipage} \hfill
  \begin{minipage}{.245\linewidth}
 \includegraphics[width=1.\textwidth]{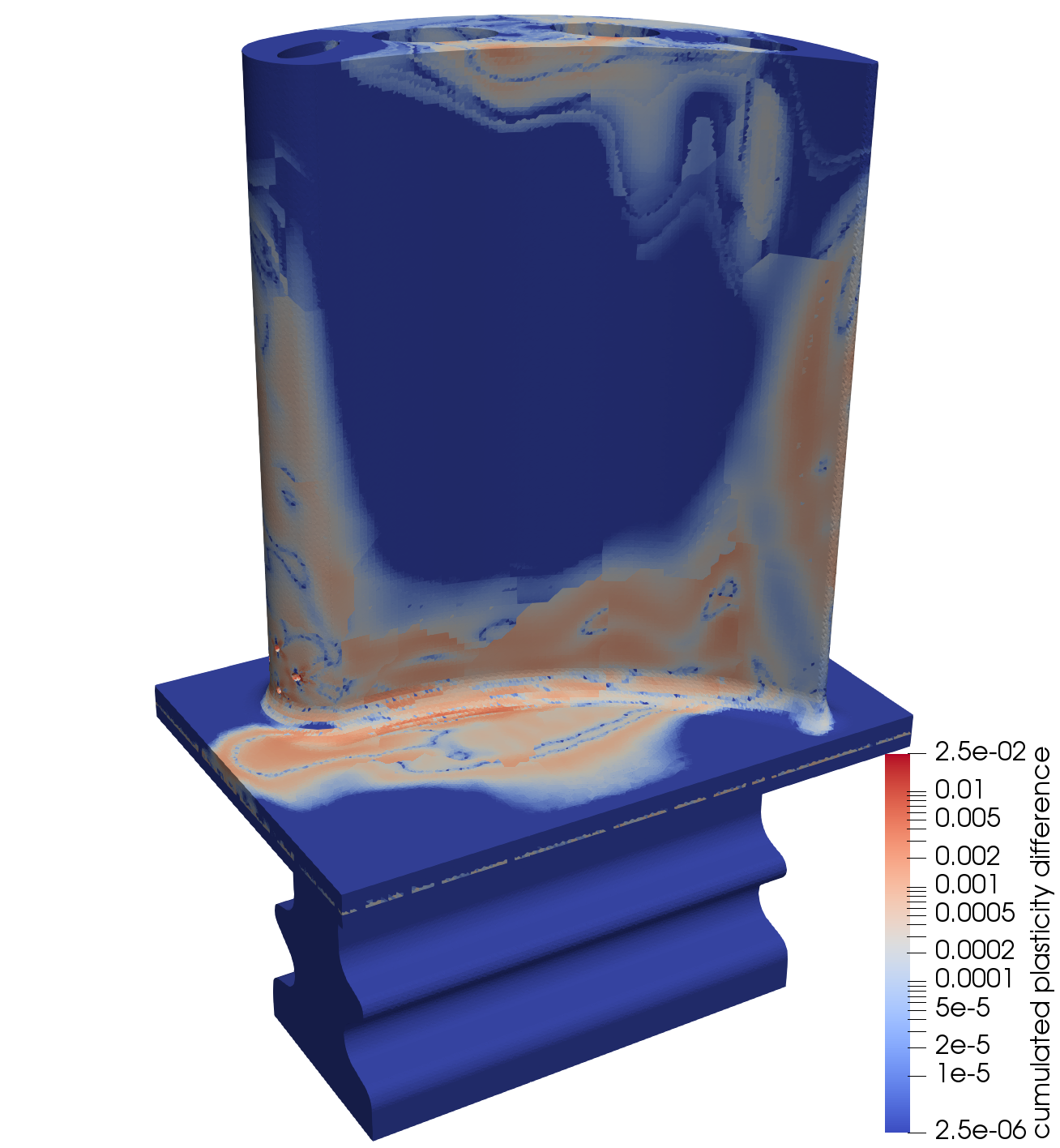}\\
   d) cumul. plast. error $10^{\rm th}$ cycle
   \end{minipage}
  \caption{Difference maps between the ROM and the HFM (divided by the largest reference value on the map): displacement at the middle of a) the $1^{\rm st}$ and b) $10^{\rm th}$ cycles, cumulated t the end of c) the $1^{\rm st}$ and d) $10^{\rm th}$ cycles.}
  \label{fig:diffmaps}
\end{figure}

\begin{figure}[h]
  \centering
  \includegraphics[width=0.85\textwidth]{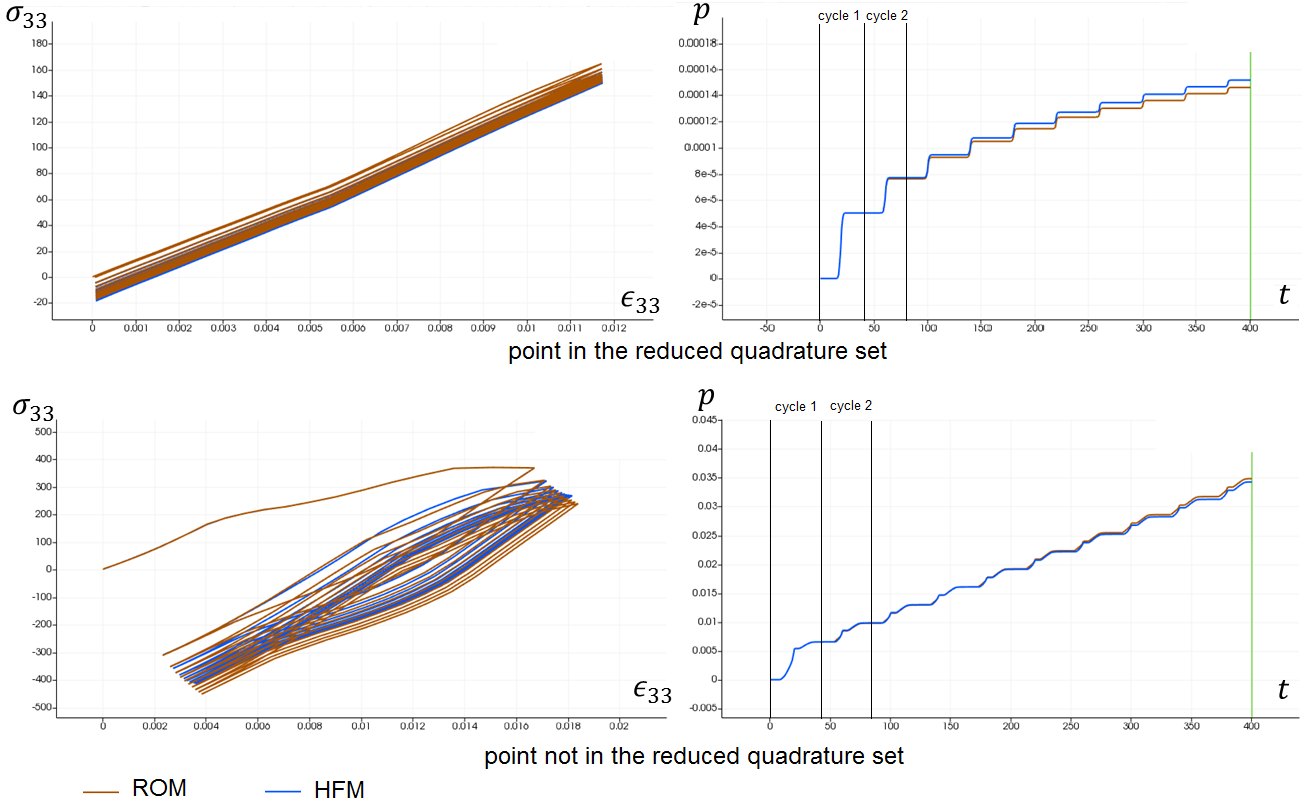}
  \caption{Comparison of reconstructed dual quantities using Algorithm~\ref{algo:GappyPOD} between the HFM (blue) and the ROM (brown) on the first 10 cycles, for a point in the reduced integration points set (top line) and not in this set (bottom line). Left-hand side: cyclic plots of $\sigma_{33}$ with respect to $\epsilon_{33}$, right-hand site: cumulated plasticity with respect to time.}
  \label{fig:res5}
\end{figure}

\section*{Conclusion and outlook}
\label{sec:outlook}

In this work, we introduce a framework to reduce nonlinear structural mechanics problems, in a nonintrusive fashion and in parallel with distributed memory. Algorithmic choices have been justified to obtain an efficient complete procedure, with a particular attention to optimize the \textit{offline} stage as well. The nonintrusive features of the framework are illustrated by reducing an Ansys Mechanical computation, and its large scale applicability is illustrated on a problem of industrial interest by constructing a reduced order model to extrapolate a strong cyclic loading on a 5 million dofs elastoviscoplastic high-pressure turbine blade on 48 processors.

In this second application, the accuracy can be further increased during the extrapolated cycles by, for instance, running the high-fidelity model each time some error indicator becomes too large, and updating the POD basis using the incremental POD. The heuristic used in the operator compression step was chosen to favor scalability; it is desirable to also guarantee some accuracy properties. Besides, the current implementation requires to load in memory, locally in each subdomain, the high-fidelity data for all time steps. Even if we can always increase the number of computer nodes to increase the available memory, it is desirable to devise memory-free strategies, where the data is not simultaneously loaded in memory for all the time steps. This would enable to reduce computations with a very large number of snapshots.

\bibliographystyle{plain}
\bibliography{biblio}

\end{document}